\documentclass[11pt]{amsart}
\usepackage{amsmath, amssymb, tensor, amsthm, amsfonts, marginnote, tikz-cd, mathtools, stmaryrd, enumitem}
\usepackage[mathscr]{euscript}
\usepackage[hidelinks]{hyperref}
\setlist[enumerate]{itemsep=0.25em}
\usepackage[skip=2.5pt, indent=10pt]{parskip}

\newcommand\cN{\mathscr{N}}
\newcommand\cO{\mathscr{O}}

\def\bC{\mathbf{C}}

\def\bF{\mathbf{F}}

\def\bP{\mathbf{P}}

\def\bZ{\mathbf{Z}}

\newcommand\frb{\mathfrak{b}}
\newcommand\frg{\mathfrak{g}}
\newcommand\frh{\mathfrak{h}}
\newcommand\fri{\mathfrak{i}}
\newcommand\frj{\mathfrak{j}}
\newcommand\frn{\mathfrak{n}}
\newcommand\frp{\mathfrak{p}}
\newcommand\frq{\mathfrak{q}}

\newcommand\frsl{\mathfrak{sl}}

\def\sfG{\mathsf{G}}

\def\sfS{\mathsf{S}}

\def\sfZ{\mathsf{Z}}

\newcommand{\beq}{\begin{equation}}
\newcommand{\eeq}{\end{equation}}

\theoremstyle{plain}
\newtheorem{theorem}{Theorem}
\newtheorem{proposition}{Proposition}
\newtheorem{lemma}{Lemma}
\newtheorem{corollary}{Corollary}

\theoremstyle{definition}

\newtheorem{remark}{Remark}
\newtheorem{example}{Example}

\DeclareMathOperator{\Shv}{Shv}
\DeclareMathOperator{\Bun}{Bun}
\DeclareMathOperator{\id}{id}
\DeclareMathOperator{\Spec}{Spec}
\DeclareMathOperator{\Hom}{Hom}

\DeclareMathOperator{\Nilp}{Nilp}
\DeclareMathOperator{\Fl}{Fl}
\DeclareMathOperator{\Gr}{Gr}

\DeclareMathOperator{\St}{St}
\DeclareMathOperator{\glob}{out}

\DeclareMathOperator{\Rep}{Rep}

\DeclareMathOperator{\Perv}{Perv}
\DeclareMathOperator{\Hecke}{Hecke}
\DeclareMathOperator{\thick}{thick}

\DeclareMathOperator{\irreg}{irreg}
\DeclareMathOperator{\reg}{reg}

\DeclareMathOperator{\Map}{Map}

\DeclareMathOperator{\Ad}{Ad}
\DeclareMathOperator{\Av}{Av}

\DeclareMathOperator{\Lie}{Lie}
\DeclareMathOperator{\pt}{pt}

\DeclareMathOperator{\cpt}{cpt}
\DeclareMathOperator{\Gait}{Gait}

\DeclareMathOperator{\SSup}{SS}

\DeclareMathOperator{\GL}{GL}

\DeclareMathOperator{\aff}{aff}
\DeclareMathOperator{\gen}{gen}
\DeclareMathOperator{\Conf}{Conf}
\DeclareMathOperator{\Vect}{Vect}
\DeclareMathOperator{\Aut}{Aut}\DeclareMathOperator{\Spf}{Spf}

\title{The tilting property of Whittaker averaged central sheaves}
\author{Jeremy Taylor}
\begin{document}
\begin{abstract} 
We characterize the kernel of Iwahori--Whittaker averaging in microlocal terms. Applying this to automorphic sheaves, we generalize Færgeman and Raskin's theorem that anti-temperedness is equivalent to having irregular singular support. Moreover, using a Radon transform argument of Bezrukavnikov and Morton-Ferguson, we extend  the tilting property of Whittaker averaged central sheaves to integer coefficients.
\end{abstract}
\maketitle
\setcounter{tocdepth}{1}
\tableofcontents
\section{Introduction}
For a space acted on by the loop group, we study the averaging functor from spherical equivariant sheaves to the Iwahori--Whittaker category. We characterize its kernel in terms of singular support and the moment map.

Applying this to automorphic sheaves (using the geometry of the global nilpotent cone to show that the resulting singular support condition is independent of the choice of point), we generalize a theorem of Færgeman and Raskin to cover integer coefficients and tame ramification. Namely we prove that
\begin{enumerate}
\item\label{Result1} an automorphic sheaf is killed by Whittaker averaging at an unramified point of the curve if and only if it has irregular singular support.
\end{enumerate} 
Unlike \cite{FR25}, we use neither theorems of Losev nor the global Whittaker coefficients functor.

We further prove that
\begin{enumerate}\setcounter{enumi}{1}
\item\label{Result2} if an automorphic sheaf is killed by Whittaker averaging at every unramified point, then it is also killed by Whittaker averaging at every tamely ramified point,
\item\label{Result3} given a tilting representation of the Langlands dual group, the Hecke action of the corresponding Whittaker averaged central sheaf at a tamely ramified point is t-exact.
\end{enumerate}
Applying \eqref{Result3} to the affine Hecke category we prove that
\begin{enumerate}
\setcounter{enumi}{3}
\item\label{Result4} given a tilting representation, the corresponding Whittaker averaged central sheaf is tilting.
\end{enumerate}

This extends a theorem of \cite{AB09, BRR20} to cover universal monodromic sheaves with integer coefficients. (Their proof uses that certain quasi-minuscule tilting representations are simple, so it does not cover some coefficient fields of very small characteristic.)

In fact we prove the stronger result that
\begin{enumerate}\setcounter{enumi}{4}
\item\label{Result5} given a standard (respectively costandard) representation, the corresponding Whittaker averaged central sheaf admits a standard (respectively costandard) filtration.
\end{enumerate}
See Remark \ref{FiniteExact} for an application of \eqref{Result5}, for which \eqref{Result4} does not seem to suffice.

The tilting property \eqref{Result4} is important in the proof of Bezrukavnikov's equivalence \cite{AB09, Be16}, where it is used to prove the vanishing of certain higher Exts. 

The result that \eqref{Result3} in genus zero with two tamely ramified points implies \eqref{Result4} is due to Bezrukavnikov and Morton-Ferguson \cite{BMF24}.
For characteristic zero field coefficients, they prove \eqref{Result3} in genus zero with two tamely ramified points, but they use Bezrukavnikov's equivalence in the process, so they do not obtain a new proof of \eqref{Result4}. (Rather they generalize \eqref{Result4} to cover all convolution exact sheaves, not just those coming from Gaitsgory's central functor.) 

Thus our contribution is to prove \eqref{Result3} for integer coefficients and arbitrary curves without assuming Bezrukavnikov's equivalence.

\subsection*{Acknowledgements}
I benefited from many conversations with Gurbir Dhillon and David Nadler. I thank Dhillon for suggesting Lemma \ref{ConnectedFibers}. I thank Joakim Færgeman for answering my questions and Sam Raskin for comments on a draft. I am grateful to Mark Macerato and Calder Morton-Ferguson for helpful discussions. I was supported by NSF MSPRF 2503562.

\section{Singular support estimates}\label{SectionEstimates}
Here we show that singular support behaves well with respect to moment maps and convolution. Moreover we provide an isolated intersection condition which can be used to bound singular support from below.

\subsection*{Weakly constructible sheaves}
All algebraic geometry is over the complex numbers.

For a smooth variety $Y$, write $\Shv(Y)$ for the weakly constructible derived category of sheaves with integer coefficients. This category is typically not cocomplete. We work in the analytic topology. The term `weakly constructible' (see Definition 8.5.6 of \cite{KS90}) means locally constant along some stratification, but without imposing any finiteness conditions on stalks.

For a smooth locally of finite type algebraic stack $Y$, the derived category of all sheaves is defined as a limit under pullback over varieties mapping smoothly to $Y$. Let $\Shv(Y)$ be the full subcategory of objects whose pullback along every such map is weakly constructible.

Let $\Perv(Y)$ be the abelian subcategory of perverse sheaves. Following the conventions of \cite{KS90}, we impose no finiteness conditions on their stalks.

For $A \in \Shv(Y)$, write $\SSup(A) \subset T^*Y$ for its singular support. As explained in Appendix F.6 of \cite{AGKRRV20}, this notion extends to algebraic stacks because singular support behaves well with respect to smooth pullback.

\subsection*{Cartesian products}
The lemma says that singular support behaves well with respect Cartesian products. The argument is similar to Section 3.3 of \cite{FR25}.

Let $f: X \rightarrow Y$ be a map of smooth varieties.

\begin{lemma}\label{SSProduct}
Let $Z$ be a smooth variety and $\Lambda \subset T^*Z$ be a closed conical subset.
\begin{enumerate}[label=(\alph*)]
\item\label{SSProduct1} If $B \in \Shv(Z \times Y)$ satisfies $\SSup(B) \subset \Lambda \times T^*Y$, then $\SSup((\id \times f)^*B) \subset \Lambda \times T^*X$.
\item\label{SSProduct2} If $A \in \Shv(Z \times X)$ satisfies $\SSup(A) \subset \Lambda \times T^*X$, then $\SSup((\id \times f)_*A) \subset \Lambda \times T^*Y$.
\end{enumerate}
Similar holds for the $!$-functors.
\end{lemma}
\begin{proof}
First we prove \ref{SSProduct1}. Factor $f$ as a closed embedding followed by a smooth map. Indeed we can use the inclusion of the graph $X \rightarrow X \times Y$ followed by the projection $X \times Y \rightarrow Y$.
For the smooth map, use Proposition 5.4.5 of \cite{KS90}. For the closed embedding, use Corollary 6.4.4 and Proposition 6.2.4(a) of \cite{KS90}. (If it is an embedded submanifold then Remark 6.2.8(i) of \cite{KS90} is simpler to parse.)

Now we prove \ref{SSProduct2}. First we factor $f$ as a composition of proper maps to smooth varieties and open embeddings into smooth varieties. Indeed Nagata's compactification theorem implies that $X$ is an open subvariety of a proper variety $\overline{X}$. By resolution of singularities we may assume that $\overline{X}$ is also smooth. Then we have the factorization $X \rightarrow X \times Y \rightarrow \overline{X} \times Y \rightarrow Y$.

For the proper maps, use Proposition 5.4.4 of \cite{KS90}. 
The open embedding can be factored as a sequence of open embeddings, such that the complement of each is an embedded submanifold.
For each such embedding, use Corollary 6.3.2 and Remark 6.2.8(ii) of \cite{KS90}.
\end{proof}

\subsection*{Moment maps}
The following lemma says that the Lagrangian correspondence between cotangent bundles induced by an equivariant map is compatible with the moment maps.

If $H$ is an algebraic group acting on $X$, let $\mu_X: T^*X \rightarrow \frh^*$ denote the moment map, taking values in the dual Lie algebra.

\begin{lemma}\label{MomentCorrespond}
Let $f : X \rightarrow Y$ be an $H$-equivariant map between smooth varieties. Then the following square commutes
\[\begin{tikzcd} X \times_Y T^*Y \arrow[d, "a"'] \arrow[r, "b"] & \arrow[d, "\mu_Y"] T^*Y  \\
T^*X \arrow[r, "\mu_X"']& \frh^*. \end{tikzcd}\] 
\end{lemma} 
\begin{proof}
For $v \in \frh$, let $v_x \in T_xX$ and $v_y \in T_yY$ be the induced tangent vectors. Since $f$ is $H$-equivariant, we have $f_* v_x = v_{f(x)}$ in $T_{f(x)}Y$. Therefore, for $(x, \eta) \in X \times_Y T^*Y$ and $v \in \frh$, we have
\[\langle \mu_X a(x, \eta), v \rangle = \langle f^* \eta, v_x \rangle = \langle \eta, f_* v_x \rangle = \langle \eta, v_{f(x)} \rangle = \langle \mu_Yb(x, \eta), v\rangle.\]
This gives the desired $\mu_X a = \mu_Y b$.
\end{proof}

The following lemma says singular support behaves well with respect to moment maps. (In particular it implies that pushforward and pullback along $H$-equivariant maps preserve the property of weak $H$-constructibility.)

\begin{lemma}\label{EquivariantSS}
Let $f : X \rightarrow Y$ be an $H$-equivariant map between smooth varieties.
\begin{enumerate}[label=(\alph*)]
\item\label{EquivariantSS1} If $f$ is smooth and surjective, then $\mu_X(\SSup (f^* B)) = \mu_Y(\SSup (B))$.
\item\label{EquivariantSS2} In general, $\mu_X(\SSup (f^*B)) \subset \mu_Y(\SSup (B))$.
\item\label{EquivariantSS3} In general, $\mu_Y(\SSup (f_* A)) \subset \mu_X(\SSup (A))$.
\end{enumerate}
Similar holds for the $!$-functors.
\end{lemma}
\begin{proof}
First we prove \ref{EquivariantSS1}. Singular support behaves well with respect to smooth pullback (see Proposition 5.4.5 of \cite{KS90}). Therefore Lemma \ref{MomentCorrespond} implies the desired \ref{EquivariantSS1}.

We will prove the remaining parts by pulling back along the action maps $a_Y$ and $a_X$. There is a Cartesian square
\[\begin{tikzcd}
H \times X \arrow[r, "\id \times f"] \arrow[d, "a_X"'] & H \times Y \arrow[d, "a_Y"] \\
X \arrow[r, "f"'] & Y.
\end{tikzcd}\]

Now we prove \ref{EquivariantSS2}.
Part \ref{EquivariantSS1} implies $\SSup(a_Y^* B) \subset \mu_H^{-1}(\mu_Y(\SSup (B))) \times T^*Y$. Lemma \ref{SSProduct}\ref{SSProduct1} implies \[\SSup(a_X^*f^* B) = \SSup((\id \times f)^* a_Y^*B) \; \subset \; \mu_H^{-1}(\mu_Y(\SSup (B))) \times T^*X.\]
Therefore \ref{EquivariantSS1} implies the desired $\mu_X(\SSup (f^*B)) \subset \mu_Y(\SSup (B))$.

Finally we prove \ref{EquivariantSS3}. Part \ref{EquivariantSS1} implies $\SSup(a_X^* A) \subset \mu_H^{-1}(\mu_X(\SSup (A))) \times T^*X$. 
Lemma \ref{SSProduct}\ref{SSProduct2} implies \[\SSup(a_Y^*f_* A) = \SSup((\id \times f)_* a_X^* A) \; \subset \; \mu_H^{-1}(\mu_X(\SSup (A))) \times T^*Y.\]
Therefore \ref{EquivariantSS1} implies the desired $\mu_Y(\SSup (f_*A)) \subset \mu_X(\SSup (A))$.
\end{proof}

\subsection*{Convolution} The following proposition says that singular support behaves well with respect to convolution.

Let $H$ act on $Y$ and let $P, Q \subset H$ be subgroups.
We will work in one of the following two settings.

\begin{enumerate}[label=(\alph*)]
\item\label{SSConvolve1} For the finite dimensional setting, $H$ is an algebraic group and $Y$ is a smooth variety. We do not assume that $H/P$ is proper, and consider either the $!$ or $*$-version of convolution.

(This will be important in the proof of Theorem \ref{AffineKernel}.)

\item\label{SSConvolve2} For the infinite dimensional setting, $H$ is a group ind-scheme and $Y$ is a pro-smooth scheme. Moreover $P$ and $Q$ are pro-algebraic group schemes such that $P \backslash Y$ and $Q \backslash Y$ are smooth varieties. We assume that $H/P$ is ind-proper.

(This is not really necessary in this paper, it is only used in the proof of Proposition \ref{NadlerYun} to explain a theorem of \cite{NY19a}.)
\end{enumerate}
Denote the moment maps by \[\mu_P:T^*(P \backslash Y)\rightarrow \frp^{\perp}/P \rightarrow \frh^*/H \quad \text{and} \quad \mu_Q:T^*(Q\backslash Y)\rightarrow \frq^{\perp}/Q \rightarrow \frh^*/H.\]

\begin{proposition}\label{SSConvolve}
If $B \in \Shv(Q \backslash H / P)$ and $A \in \Shv(P \backslash Y)$ then \beq \label{SSConvolveInclusion} \mu_Q(\SSup(B*A)) \; \subset \; \mu_P(\SSup(A)).\eeq
\end{proposition}
\begin{proof}
First we consider the finite dimensional setting \ref{SSConvolve1}. 
By Lemma \ref{EquivariantSS}\ref{EquivariantSS1} it suffices to consider the case $Y = H \times X$ with $H$ acting trivially on $X$. Denote the moment maps for the right $H$-actions by \[\nu_P: T^*(P \backslash H) \rightarrow \frh^* \quad \text{and} \quad \nu_Q: T^*(Q \backslash H) \rightarrow \frh^*.\]
Convolution is given by pullback then pushforward along 
\[Q \backslash H/P \times P \backslash H \times X  \leftarrow Q \backslash H \times^P H \times X \rightarrow Q \backslash H \times X.\]
These maps are equivariant for the right $H$-action, therefore Lemma \ref{EquivariantSS} implies that $\nu_Q(\SSup(B*A)) \subset \nu_P(\SSup(A))$.

The moment maps for the left and right actions of a group on itself differ by the adjoint action on the dual Lie algebra. Therefore the following diagrams commute
\beq\label{LeftRightTriangle} \begin{tikzcd}
T^*(P \backslash H) \arrow[r, "\nu_P"] \arrow[d, "\mu_P"'] & \frh^* \arrow[dl] & & T^*(Q \backslash H) \arrow[r, "\nu_Q"] \arrow[d, "\mu_Q"'] & \frh^* \arrow[dl]\\
\frh^*/H  & & &  \frh^*/H, & 
\end{tikzcd}\eeq
giving the desired \eqref{SSConvolveInclusion}.

Now we consider the infinite dimensional setting \ref{SSConvolve2}. The above argument does not work directly, because it would involve sheaves on infinite type schemes. Instead we use the additional ind-properness assumption to argue at the level of cotangent bundles.

The convolution diagram
\[Q \backslash H/P \times P \backslash Y  \leftarrow Q \backslash H \times^P Y \rightarrow Q \backslash Y\]
induces a Lagrangian correspondence between cotangent bundles
\[T^*(Q \backslash H/P \times P \backslash Y) \xleftarrow{a} T^*(Q \backslash H/P \times P \backslash Y)|_{Q \backslash H \times^P Y} \underset{T^*(Q \backslash H \times^P Y)}{\times} T^*(Q \backslash Y)|_{Q \backslash H \times^P Y}\xrightarrow{b} T^*(Q \backslash Y).\]

We claim that for any closed conic subset $V \subset T^*(P \backslash Y)$ we have \beq \label{MuQMuPCotangent}\mu_Q(b(a^{-1}(T^*(Q \backslash H/P) \times V))) \; \subset \; \mu_P(V).\eeq
Indeed by Lemma \ref{MomentCorrespond} and commutavity of
\[\begin{tikzcd}
Q \backslash H/P \times P \backslash H \times Y \arrow[d] & \arrow[l]  Q \backslash H \times^P H \times Y \arrow[r] \arrow[d] & Q \backslash H \times Y \arrow[d] \\
Q \backslash H/P \times P \backslash Y & \arrow[l]  Q \backslash H \times^P Y \arrow[r] & Q \backslash Y,
\end{tikzcd}\]
it suffices to consider the case $Y = H \times X$ with $H$ acting trivially on $X$. Then Lemma \ref{MomentCorrespond} and commutativity of \eqref{LeftRightTriangle} imply the desired \eqref{MuQMuPCotangent}.

Because singular support behaves well with respect to pro-smooth pullback and ind-proper pushforward, \eqref{MuQMuPCotangent} implies the desired \eqref{SSConvolveInclusion}.
\end{proof}

\subsection*{Isolated intersection condition}
Having provided many bounds from above on singular support, we now also provide one bound from below. Namely the following proposition gives an isolated intersection condition under which a covector is contained in the singular support of the pushforward of a sheaf.

Let $f: X \rightarrow Y$ be a smooth map between smooth varieties. It induces a Lagrangian correspondence between cotangent bundles \[T^*X \xleftarrow{a} X \times_Y T^*Y \xrightarrow{b} T^*Y.\]

\begin{proposition}\label{Below}
Let $A \in \Shv(X)$ and $(x, \xi) \in X \times_Y T^*Y$ be a smooth point of $\SSup(A)$. Then the following are equivalent
\begin{enumerate}[label=(\alph*)]
\item\label{Below1} $(x, \xi)$ is an isolated point in $a^{-1}(\SSup(A)) \cap b^{-1}(\xi)$ and the intersection is clean there,
\item\label{Below2} $X \times_Y T^*Y$ intersects $\SSup(A)$ transversely at $(x, \xi)$. 
\end{enumerate}
If the above holds then $\xi \in \SSup(f_!A)$ and $\xi \in \SSup(f_* A)$.
\end{proposition}
\begin{proof}
The equivalence between \ref{Below1} and \ref{Below2} follows because
\begin{enumerate}
\item[-] $\SSup(A) \subset T^*X$ is Lagrangian by Theorem 8.5.5 of \cite{KS90},
\item[-] $T_{(x, \xi)} (b^{-1}(\xi))$ is the orthogonal to $T_{(x, \xi)} (X \times_Y T^*Y)$ under the symplectic form on $T_{(x, \xi)}(T^*X)$.
\end{enumerate}
Corollary 7.5.12 of \cite{KS90} says that \ref{Below2} implies $\xi \in \SSup(f_*A)$, but we also provide an alternative argument below.

If $A \in \Shv(X)$ satisfies \ref{Below1}, then there exists a function $g: Y \rightarrow \bC$ such that 
\begin{enumerate}[label=(\roman*)]
\item\label{BelowStep1} the graph of its differential $\Gamma_{d(gf)} \subset T^*X$ intersects $\SSup(A)$ transversely at $(x, \xi)$.
\end{enumerate}
Let $M \coloneqq (gf)^{-1}(0) \subset X$ and $N \coloneqq g^{-1}(0) \subset Y$. Let $\mu_M A \in \Shv(T^*_MX)$ be the microlocalization to the conormal bundle.

Lemma \ref{Microstalk} and \ref{BelowStep1} imply that the vanishing cycles $\phi_{gf, x} A$ is nonzero. Therefore Proposition 8.6.3 of \cite{KS90} gives that 
\begin{enumerate}[label=(\roman*)]
\setcounter{enumi}{1}
\item\label{BelowStep1.5} the stalk at $(x, \xi)$ of $\mu_M A$ is nonzero.
\end{enumerate}
Moreover Corollary 5.4.10 of \cite{KS90} says that the support of $\mu_M A$ is contained in $T^*_M X \cap \SSup(A)$. Hence \ref{Below1} implies that
\begin{enumerate}[label=(\roman*)]
\setcounter{enumi}{2}
\item\label{BelowStep2} $(x, \xi)$ is isolated in the intersection of $b^{-1}(\xi)$ with the support of $a^* \mu_M A$.
\end{enumerate}

We may assume that $f$ is proper, because \ref{Below1} is local on $X$ around $x$. Therefore Proposition 4.3.4 of \cite{KS90} gives \beq\label{MicroTransport} \mu_N f_! A \simeq \mu_N f_* A \simeq b_!a^* \mu_M A.\eeq Moreover \ref{BelowStep1.5} and \ref{BelowStep2} imply that $\xi$ belongs to the support of \eqref{MicroTransport}. Corollary 5.4.10 of \cite{KS90} gives the desired $\xi \in \SSup(f_! A)$ and $\xi \in \SSup(f_* A)$.
\end{proof}

\begin{remark}
Theorem 20.1.3 of \cite{AGKRRV20} gives a stronger form of Proposition \ref{Below} in which the intersection in \ref{Below1} is not assumed clean, only isolated. However they assume complex coefficients and Zariski constructibility, because their proof uses D-modules on schemes of infinite type over the complex numbers.
\end{remark}

We used the following lemma. Under its hypotheses, the vanishing cycles functor is said to calculate `microstalk'.

\begin{lemma}\label{Microstalk}
Let $A \in \Shv(X)$ be a sheaf on a smooth variety. Let $f: X \rightarrow \bC$ be a function such that the graph of its differential $\Gamma_{df}$ intersects the singular support transversely at a smooth point $(x, \xi) \in \SSup(A)$. Then $\phi_{f, x}A$ (the stalk at $x$ of vanishing cycles for the function $f$) is nonzero.
\end{lemma}
\begin{proof}
Follows by the proof of Proposition 7.5.3 of \cite{KS90}.
\end{proof}

\section{Finite Whittaker averaging}
For an algebraic stack with $G$-action, there is an averaging functor that sends $N$-equivariant sheaves to Whittaker sheaves. We show that its kernel consists of sheaves such that the image under the moment map of their singular support is irregular.

This extends Theorem 3.1.2.1 of \cite{FR25} (see also Theorem 3.0.3.1 of \cite{DF25}) to allow weakly constructible sheaves with integer coefficients. 
Whereas they use Beilinson--Bernstein localization and theorems of Losev \cite{Lo10, Lo11, Lo17}, we use only singular support estimates from \cite{KS90}.

\subsection*{The finite Hecke category}
Let $G$ be a complex reductive group. For simplicity, we will assume that it is adjoint. Let $N \subset B$ be the unipotent radical of the Borel. Let $W$ be the Weyl group.

Let $G$ act on a locally of finite type algebraic stack $Y$. Let $\mu: T^*(N \backslash Y) \rightarrow \frb/N$ be the moment map. 
\begin{enumerate}
\item[-] Let $\Shv_{\cN}(N \backslash Y)$ be the full subcategory of sheaves whose singular support is contained in $\mu^{-1}(\frn/N)$.
\item[-] We say that a sheaf in $\Shv(N \backslash Y)$ is `$G$-irregular' if its singular support is contained in $\mu^{-1}(\frb^{\irreg}/N)$. 
\end{enumerate}

Let $R \coloneqq \bZ[\Lambda]$ be the group ring of the coweight lattice. The universal monodromic finite Hecke category $\Shv_\cN(N \backslash G / N)$ is bilinear over $R$.

Let $\Delta_w$ and $\nabla_w \in \Perv_{\cN}(N \backslash G/N)$ be the standard and costandard extension of the universal local system on the orbit indexed by $w \in W$.

Convolution in the finite Hecke category is not t-exact, but it becomes so after quotienting out $G$-irregular sheaves.

\begin{lemma}\label{FiniteHecke}
Let $B \in \Shv_\cN(N \backslash G/N)$ and $A \in \Shv(N \backslash Y)$. Let $w \in W$.

\begin{enumerate}[label=(\alph*)]
\item\label{FiniteHecke3} The functors $\Delta_w * -$ and $\nabla_w * -$ are left and right t-exact respectively.
\item\label{FiniteHecke5} If $A$ is $G$-irregular, then so is each of its perverse cohomology sheaves.
\item\label{FiniteHecke1} If $B$ is $G$-irregular, then so is $B*A$.
\item\label{FiniteHecke2} There exists a wrong way map $\phi_w: \nabla_w \rightarrow \Delta_w$ whose cone is $G$-irregular.
\item\label{FiniteHecke4} If $A$ is perverse, then $\tau^{> 0}(\Delta_w * A )$ is $G$-irregular.
\end{enumerate}
\end{lemma}
\begin{proof}
Part \ref{FiniteHecke3} follows by Proposition 10.3.17 of \cite{KS90}. Part \ref{FiniteHecke5} follows by the t-exactness of vanishing cycles, see Corollary 10.3.14 of \cite{KS90}.  Part \ref{FiniteHecke1} follows by Lemma \ref{EquivariantSS} and left $G$-equivariance of \[G/N \times N \backslash Y  \leftarrow G \times^N Y \rightarrow Y.\]

Now we prove \ref{FiniteHecke2} when $w = s$ is a simple reflection. Let $\alpha$ be the corresponding simple coroot and $P_s$ be the corresponding parabolic subgroup. There exists a local system $K_s \in \Perv_\cN(N \backslash P_s/N)$ with $*$-stalks $R/(e^{\alpha} - 1)[1]$ and short exact sequences \[0 \rightarrow \Delta_s \rightarrow \Delta_s \rightarrow \Delta_s/(e^{\alpha} - 1) \rightarrow 0 \quad \text{and} \quad 0 \rightarrow \Delta_1/(e^{\alpha} - 1) \rightarrow \Delta_s/(e^{\alpha} - 1) \rightarrow K_s \rightarrow 0.\] Equation (4.1) of \cite{Ta25} implies that $\Delta_1/(e^{\alpha} - 1)$ is the $!$-restriction of $\Delta_s$ to $N \backslash B/N$. The composition $\Delta_s \rightarrow \Delta_s/(e^{\alpha} - 1) \rightarrow K_s$ induces an isomorphism on $!$-fibers along $N \backslash B/N$. Therefore it fits into a short exact sequence \[0 \rightarrow \nabla_s \xrightarrow{\phi_s} \Delta_s \rightarrow K_s \rightarrow 0.\]
Since $K_s$ is weakly $P_s$-constructible, it is $G$-irregular.

Now we prove \ref{FiniteHecke2} in general by induction on $\ell(w)$. Write $w = uv$ such that $\ell(w) = \ell(u) + \ell(v)$ and both $u, v \neq 1$. Let $\phi_w$ be the composite \[\nabla_u * \nabla_v \xrightarrow{\nabla_u * \phi_v} \nabla_u * \Delta_v \xrightarrow{\phi_u * \Delta_v} \Delta_u * \Delta_v.\]
By induction and \ref{FiniteHecke1} the cones of $\nabla_u * \phi_v$ and $\phi_u * \Delta_v$ both are $G$-irregular. Therefore the cone of $\phi_w$ is $G$-irregular.

Finally we prove \ref{FiniteHecke4}.
By \ref{FiniteHecke2} there is a triangle  \[\nabla_w * A \xrightarrow{\phi_w * A} \Delta_w * A \rightarrow K_w * A,\]
in which $K_w$ is $G$-irregular. Moreover \ref{FiniteHecke1} and \ref{FiniteHecke5} imply that all perverse cohomology sheaves of $K_w * A$ are also $G$-irregular.
Finally \ref{FiniteHecke3} implies that the truncation $\tau^{> 0}(\Delta_w * A)$ is $G$-irregular.
\end{proof}

\subsection*{Kernel of finite Whittaker averaging}
Here we characterize the kernel of finite Whittaker averaging in terms of singular support. 

Let $\Xi \in \Shv_{\cN}(N \backslash G/N)$ be the big tilting sheaf. This universal monodromic version is constructed in \cite{LNY24, Ta25}.

\begin{theorem}\label{GeneralKernel}
For $A \in \Shv_{\cN}(N \backslash Y)$ the following are equivalent
\begin{enumerate}[label=(\alph*)]
\item\label{GeneralKernel2} $A$ is $G$-irregular,
\item\label{GeneralKernel1} $\Xi * A \simeq 0$, i.e. $A$ is killed by Whittaker averaging.
\end{enumerate}
\end{theorem}
\begin{proof}
If $x \in G/N$, let $N_x \coloneqq \Ad_x N$ and $\delta_x \in \Shv(N_x \backslash G/N)$ be the skyscraper at $x$.

First we prove that \ref{GeneralKernel2} implies \ref{GeneralKernel1} in the case $Y = G$. The strategy is reduction to the case $Y = G / N_x$ by using singular support estimates applied to averaging along \[\Av_{N_x}: N \backslash G \rightarrow N \backslash G/N_x.\]

Let $A \in \Shv(N \backslash G)$ be $G$-irregular.

\begin{enumerate}[label=(\roman*)]
\item\label{Step1} The big tilting sheaf is self dual. (This follows by the monoidal equivalence between the Hecke category and Soergel bimodules.)
\item\label{Step2} If $B \in \Shv_\cN(N \backslash G / N_x)$ then $\Hom(\Xi * \delta_x , B)$ calculates vanishing cycles of $B$ at a regular covector by Definition 3.2 of \cite{Ta25}. (Similar ideas in the global setting are discussed in \cite{NT24}.)
\item\label{Step3} Lemma \ref{EquivariantSS}\ref{EquivariantSS3} implies that $\Av_{N_x*}A$ is $G$-irregular.
\end{enumerate}
Therefore, for all $x \in G/N$ we have \[\Hom(\delta_x, \Xi*A) \simeq \Hom(\delta_x, \Av_{N_x*}(\Xi * A)) \overset{\ref{Step1}}{\simeq} \Hom(\Xi * \delta_x, \Av_{N_x*} A) \overset{\ref{Step2}, \ref{Step3}}{\simeq} 0.\] 
Hence $\Xi * A \simeq 0$ as desired. (A weakly constructible sheaf is zero if its $!$-stalks all vanish.)

Now we prove that \ref{GeneralKernel2} implies \ref{GeneralKernel1} when $Y$ is general. Let $A \in \Shv(N \backslash Y)$ be $G$-irregular. For every equivariant map $i: N \backslash G \rightarrow N \backslash Y$,  Lemma \ref{EquivariantSS}\ref{EquivariantSS2} implies that $i^* A$ is also $G$-irregular. Therefore $i^* (\Xi * A) \simeq \Xi * (i^* A) \simeq 0$ vanishes by the case $Y = G$ that was settled above. Hence $\Xi * A \simeq 0$ because all of its $*$-stalks vanish.

Finally we prove that \ref{GeneralKernel1} implies \ref{GeneralKernel2}. Our proof uses t-exactness of vanishing cycles and it is similar to Proposition 14.3.2 of \cite{FG06}. Suppose for contradiction that the singular support of $A \in \Shv_{\cN}(N \backslash G)$ contains a covector $\xi \in \SSup(A)$ such that $\mu(\xi) \in \frn^{\reg}/N$ is regular. 

Proposition 5.1.3 of \cite{KS90} implies that $\xi$ is contained in the singular support of one of the perverse cohomology sheaves of $A$. Proposition 3.13 of \cite{Ta25} and Lemma \ref{FiniteHecke}\ref{FiniteHecke3} imply that $\Xi * -$ is t-exact. Therefore it suffices to assume that $A$ is perverse.

Let $\Lambda \subset T^*(N \backslash Y)$ be the union over $w \in W$ of $\SSup(\Delta_w * A)$. Since the regular locus is open, we may assume that $\xi$ is a smooth point of $\Lambda$. Therefore we can choose a vanishing cycles functor \[\phi: \Shv_\Lambda(N \backslash Y) \rightarrow \Vect \quad \text{computing microstalk at } \xi.\]
It is t-exact by Corollary 10.3.13 of \cite{KS90}.

Since $A$ has nilpotent singular support, $\Delta_1 * A \simeq A$. Hence Lemma \ref{Microstalk} gives $\phi(\Delta_1 * A) \not\simeq 0$. Moreover $\Xi$ admits a standard filtration, and Lemma \ref{FiniteHecke}\ref{FiniteHecke4} implies that $\phi(\Delta_w * A)$ is concentrated in degree zero for all $w \in W$. Therefore $\phi(\Xi * A) \not\simeq 0$ and hence $\Xi * A \not\simeq 0$.
\end{proof}

\begin{remark}\label{AvChi}
If the universal monodromic Whittaker averaging functor $\Av_{\chi}$ is defined as in \cite{DT25}, it admits a left adjoint and the resulting comonad is given by $\Xi * -$.
Since the left adjoint is conservative, $\Xi * -$ has the same kernel as $\Av_{\chi}$. As we are only interested in its kernel, we need not recall the construction of $\Av_{\chi}$, but will instead work with the big tilting sheaf.
\end{remark}

\section{Iwahori--Whittaker averaging}
For a space acted on by the loop group, we study the averaging functor from spherical equivariant sheaves to Iwahori--Whittaker sheaves. It is constructed by averaging to $J$, a certain conjugate of $I$, followed by finite Whittaker averaging.
We show that its kernel consists of sheaves such that the image under the moment map of their singular support is irregular at the closed point of the formal disc (but not necessarily at the generic point).

Having already understood the kernel of finite Whittaker averaging, it remains to understand averaging from $G_O$ to $J$ and its effect on singular support. Although this averaging functor involves pushforward along a non-proper map, we are able to bound the singular support from above using Proposition \ref{SSConvolve}. To bound the singular support from below, we will use the isolated intersection condition provided by Proposition \ref{Below}.

\subsection*{Notation}
Let $F \coloneqq \bC((t))$ and $O \coloneqq \bC[[t]]$. The topological dual of $F$ is identified with $Fdt$ by the residue pairing.

Let $G_F$ be the loop group and $G_O$ be the spherical subgroup. Let $\Gr \coloneqq G_F/G_O$ be the affine Grassmannian.

Let $I \coloneqq G_O \times_G N$ be the pro-unipotent radical of the Iwahori. (This notation is nonstandard because $I$ usually denotes the full Iwahori.) 

Let $N^-$ be the unipotent radical of the opposite Borel. Let $\rho$ be half the sum of the positive coroots, a coweight because $G$ is adjoint. Define \[J \coloneqq \Ad_{t^{-\rho}}(G_O \times_G N^-),\]
a conjugate of the pro-unipotent radical of the Iwahori.

\begin{enumerate}
\item[-] Let $\Xi' \in \Perv_{\cN}(J \backslash G_F/ J)$ be the big tilting sheaf supported on the closure of the orbit through a lift of $\Ad_{t^{-\rho}} w_0$, where $w_0 \in W$ is the longest element.
\item[-] Let $\Av_{J!}$ and $\Av_{J*}: \Shv(G_O \backslash \Gr) \rightarrow \Shv(J \backslash \Gr)$ be the averaging functors defined using left and right adjoint sheaf functors respectively.
\item[-] Let $\Psi \coloneqq \Xi' * \Av_{J!}\delta[\langle -2 \check{\rho}, \rho \rangle] \in \Perv(J \backslash \Gr)$ be obtained by averaging the monoidal unit $\delta \in \Perv(G_O \backslash \Gr)$ to be $J$-equivariant then convolving with the tilting sheaf.
\end{enumerate}

The following lemma implies that $\Psi$ can also be defined using the $*$-averaging functor.

\begin{lemma}\label{PsiLemma}
There is an isomorphism \beq\label{PsiShriekStar} \Xi' * \Av_{J!}\delta[\langle -2 \check{\rho}, \rho \rangle] \simeq \Xi' * \Av_{J*} \delta [\langle 2\check{\rho}, \rho \rangle] \quad \text{ in } \quad \Perv(J \backslash \Gr).\eeq
\end{lemma}
\begin{proof}
The cone of $\Av_{J!}\delta[\langle -2 \check{\rho}, \rho \rangle] \rightarrow \Av_{J*} \delta [\langle 2\check{\rho}, \rho \rangle]$ is supported on smaller dimensional $J$-orbits, each of which is stabilized by $\Ad_{t^{-\rho}}P^-$ for some parabolic $P^-$ strictly containing $B^-$. Therefore it is killed by convolution with $\Xi'$. This implies \eqref{PsiShriekStar}.
\end{proof}

\begin{remark}\label{KernelPsi}
Equation \eqref{AveragingTriangle} and Remark \ref{AvChi} imply that $\Psi * -$ has the same kernel as the averaging functor from spherical equivariant sheaves to the Iwahori--Whittaker category.
\end{remark}

\subsection*{Preliminaries}
Here we establish the main geometric input for the proof Theorem \ref{AffineKernel}.

Choose an $\frsl(2)$-triple $f, h, e \in \frg$ such that $f \in \frn^-$ is regular and $h$ is in the image of the derivative of the coweight $\rho$.
Let
\begin{enumerate}
\item[-] $\frg_i \subset \frg$ be the $h$-eigenspace of weight $2i$,
\item[-] $\frg_{\geq i} \subset \frg$ be the sum of the $h$-eigenspaces of weights $\geq 2i$. 
\end{enumerate}

The following lemma will be used to bound the singular support from above.

\begin{lemma}\label{RhoGIntersect}
The fiber product \beq \label{IrregCapReg} (\frg^{\irreg} + t\frg_O)dt \times_{(\frg_F/\Ad_{t^{-\rho}}\frg_O)dt/\Ad_{t^{-\rho}}G_O} \Ad_{t^{-\rho}}(t^{-1} (\frn^-)^{\reg} + \frg_O)dt \quad \text{is empty.}\eeq
\end{lemma}
\begin{proof}
The image of $\frg_O \cap \Ad_{t^{-\rho}}(t^{-1}\frg_O)$ under evaluation at the closed point $\frg_O \rightarrow \frg$ is contained in $\frg_{\geq -1}$.
Moreover since \[\frg_O \cap \Ad_{t^{-\rho}}(t^{-1}\frg^{\reg} + \frg_O) \; \subset \; \Ad_{t^{-\rho}}(t^{-1}(\frn^-)^{\reg} + \frg_O),\] 
every element in the image of \[\frg_O \cap \Ad_{t^{-\rho}}(t^{-1}\frg^{\reg} + \frg_O) \rightarrow \frg_O \rightarrow \frg\] is nonzero in every negative simple root space.
Since any such element of $\frg_{\geq -1}$ is regular, we conclude that \[\frg_O \cap \Ad_{t^{-\rho}}(t^{-1}\frg^{\reg} + \frg_O) \; \subset \; \frg^{\reg} + t\frg_O.\]

Combining this with the observation that \[\Ad_{t^{-\rho}}(t^{-1} (\frn^-)^{\reg} + \frg_O)dt \rightarrow (\frg_F/\Ad_{t^{-\rho}} \frg_O)dt/\Ad_{t^{-\rho}} G_O\] factors through \[(\Ad_{t^{-\rho}}(t^{-1} \frg^{\reg} + \frg_O)/\Ad_{t^{-\rho}}\frg_O)dt/\Ad_{t^{-\rho}}G_O\]
implies that \eqref{IrregCapReg} is empty.
\end{proof}

The following lemma will be used to check the isolated intersection condition to bound the singular support from below.

\begin{lemma}\label{RegularMatrix}
Let $\fri$ and $\frj$ be Lie algebras of $I$ and $J$ respectively.
\begin{enumerate}[label=(\alph*)]
\item\label{RegularMatrix1} The adjoint action map $G_O \times (f + \frb_O) \rightarrow (\frg^{\reg})_O = \frg^{\reg} + t\frg_O$ is surjective.
\item\label{RegularMatrix2} If $\eta \in f + \frb_O$ and $x \in \frj - (\fri \cap \frj)$, then $\Ad_x \eta \not\in \frg_O$.
\end{enumerate}
\end{lemma}
\begin{proof}
First we prove \ref{RegularMatrix1}.
Following Remark 3.1.4 of \cite{Ri17}, the set of $\eta \in f + \frb$ satisfying $\frg = \frb + \Ad_\eta \frg$ is
\begin{enumerate}
\item[-] open and also closed under the contracting $\bC^{\times}$-action $x \mapsto z^{-1} \Ad_{\rho(z)} \eta$,
\item[-] contains $f$ because $\frg = \frg^e \oplus \Ad_f \frg$. 
\end{enumerate}
Therefore $\frg = \frb + \Ad_\eta \frg$ holds for all $\eta \in f + \frb$. Hence the surjection $G \times (f + \frb) \rightarrow \frg^{\reg}$ is smooth. Therefore every $O$-point of $\frg^{\reg}$ can be lifted.

Now we prove \ref{RegularMatrix2}.
If $x \in \frj - (\fri \cap \frj)$, then \[x \; \in \; \zeta + \frg_O + (\frg_{\geq i+1})_F \quad \text{for some } i \geq 2 \text{ and } \zeta \in (\frg_i)_F - (\frg_i)_O.\] Therefore \beq\label{AdXEta} \Ad_x \eta \; \in \; \Ad_\zeta f + \frg_O + (\frg_{\geq i})_F,\eeq
where $\Ad_\zeta f \in (\frg_{i-1})_F - (\frg_{i-1})_O$ has a pole, because $\zeta$ lies in a positive $h$-eigenspace. Hence \eqref{AdXEta} is not contained in $\frg_O$.
\end{proof}

\begin{example}
If $G = \GL(3)$ then $\frj$ consists of matrices of the form \[\left(\begin{smallmatrix}*t&*&*t^{-1}\\ *t&*t&* \\ *t^2 & *t & *t\end{smallmatrix}\right) \quad \text{for} \quad * \in O.\]
If for example \[x = \left(\begin{smallmatrix}0&0&t^{-1}\\ 0&0&0 \\ 0 & 0 & 0\end{smallmatrix}\right) \; \in \; \frj - (\fri \cap \frj) \quad \text{and} \quad \eta = \left(\begin{smallmatrix}*&*&*\\ 1&*&* \\ 0 & 1 & *\end{smallmatrix}\right) \; \in \; f + \frb_O,\]
then \[\Ad_x \eta = \left(\begin{smallmatrix}0&t^{-1}&*t^{-1}\\ 0&0&-t^{-1} \\ 0 & 0 & 0\end{smallmatrix}\right) \; \not\in \; \frg_O.\]
\end{example}

\subsection*{Kernel of Iwahori--Whittaker averaging}
Here we characterize the kernel of Iwahori--Whittaker averaging by an irregularity condition on singular support at the closed point of the disc. 

Let $G_F$ act on an infinite type scheme $Y$, such that the quotient $G_O \backslash Y$ is a locally of finite type algebraic stack.
We get a moment map 
\[\mu: T^*(G_O \backslash Y) \rightarrow \frg_Odt/G_O\]
valued in Higgs fields on the formal disc.
We will also consider the moment map
\[\nu: T^*(J \backslash Y) \rightarrow \frj^{\perp}/J.\]

\begin{theorem}\label{AffineKernel}
For $A \in \Shv(G_O \backslash Y)$ the following are equivalent
\begin{enumerate}[label=(\alph*)]
\item\label{AffineKernel1} $\mu(\SSup(A)) \subset (\frg^{\irreg} + t\frg_O)dt/G_O$, i.e. the image of the singular support under the moment map consists of Higgs fields on the formal disc that are irregular at the closed point,
\item\label{AffineKernel2} $\nu(\SSup(\Av_{J!}A)) \subset \Ad_{t^{-\rho}}(t^{-1} (\frn^-)^{\irreg} + \frg_O)dt/J$,
\item\label{AffineKernel3} $\Psi * A \simeq 0$, i.e. $A$ is killed by Iwahori--Whittaker averaging. 
\end{enumerate}
\end{theorem}
\begin{proof}
Theorem \ref{GeneralKernel} implies that \ref{AffineKernel2} is equivalent to \ref{AffineKernel3}.

Now we prove \ref{AffineKernel1} implies \ref{AffineKernel2}. 
The averaging functor $\Av_{J!}$ is given by $*$-pullback then $!$-pushforward along \beq\label{GOToI} G_O \backslash Y \leftarrow (I \cap J) \backslash Y \rightarrow J \backslash Y.\eeq

Suppose that $\mu(\SSup(A)) \subset (\frg^{\irreg} + t\frg_O)dt/G_O$. 
Since singular support behaves well with respect to smooth pullback, 
\begin{enumerate}
\item[-] the image of $\SSup(A|_{(I \cap J) \backslash Y})$ under the moment map \[\mu_P: T^*((I \cap J) \backslash Y) \rightarrow (\fri \cap \frj)^{\perp}/(I \cap J) \rightarrow (\frg_F/\Ad_{t^{-\rho}}\frg_O)dt/\Ad_{t^{-\rho}}G_O\]
is contained in the image of $(\frg^{\irreg} + t\frg_O)dt$.
\end{enumerate}

Choose a congruence subgroup $K \subset G_O$ such that $\Ad_{t^{-\rho}}K \subset I \cap J$. The finite dimensional case of Proposition \ref{SSConvolve} for $H \coloneqq \Ad_{t^{-\rho}}G_O/\Ad_{t^{-\rho}}K$ with subgroups $P \coloneqq (I \cap J)/\Ad_{t^{-\rho}}K$ and $Q \coloneqq J/\Ad_{t^{-\rho}}K$ implies that 
\begin{enumerate}
\item[-] the image of $\SSup(\Av_{J!}A)$ under the moment map \[\mu_Q: T^*(J \backslash Y) \rightarrow \frj^{\perp}/J \rightarrow (\frg_F/\Ad_{t^{-\rho}}\frg_O)dt/\Ad_{t^{-\rho}}G_O\]
is contained in the image of $(\frg^{\irreg} + t\frg_O)dt$.
\end{enumerate}
Therefore Lemma \ref{RhoGIntersect} gives the desired \ref{AffineKernel2}.

Now we prove that \ref{AffineKernel2} implies \ref{AffineKernel1}. 
The Lagrangian correspondence between cotangent bundles induced by \eqref{GOToI} is obtained from the correspondence
\beq\label{RhoCorrespondence} \begin{tikzcd}
(\frg_Odt \cap \frj^{\perp})/(I \cap J) \arrow[r] \arrow[d] & \frj^{\perp}/(I \cap J) \arrow[r] \arrow[d] & \frj^{\perp}/J \\
\frg_Odt/(I \cap J) \arrow[r] \arrow[d] & (\fri \cap \frj)^{\perp}/(I \cap J) & \\
\frg_Odt/G_O & & 
\end{tikzcd}\eeq
by pullback along the moment map $G_F \backslash T^*Y \rightarrow \frg_Fdt/G_F$.

Suppose for contradiction that there exists a covector $\xi \in \SSup(A)$ such that \[\mu(\xi) \; \in \; (\frg^{\reg} + t\frg_O)dt/G_O.\]
Since regularity is an open condition, we may assume that $\SSup(A)$ is smooth at $\xi$.
\begin{enumerate}
\item[-] Lemma \ref{RegularMatrix}\ref{RegularMatrix1} implies that $\mu(\xi)$ is the image of some $\eta \in (f + \frb_O)dt \subset \frg_Odt \cap \frj^{\perp}$.
\item[-] Lemma \ref{RegularMatrix}\ref{RegularMatrix2} implies that $(\frg_Odt \cap \frj^{\perp})/(I \cap J) \times_{\frj^{\perp}/J} \eta$ contains an isolated point at which the intersection is clean, i.e. we have 
\[\begin{tikzcd}
\text{isolated point} \arrow[r] \arrow[d] & \eta \arrow[d] \\
(\frg_Odt \cap \frj^{\perp})/(I \cap J) \arrow[r] & \frj^{\perp}/J.
\end{tikzcd}\]
\end{enumerate}
Therefore Proposition \ref{Below} implies that $\eta \in \nu(\SSup(\Av_{J!} A))$.
\end{proof}

\begin{example} Averaging from $G_O$ to $J$ is controlled by the correspondence \eqref{RhoCorrespondence}. For $G = \GL(3)$ this correspondence becomes \[\left(\begin{smallmatrix} *&*&*\\ *&*&*\\ *&*&*\\ \end{smallmatrix}\right)dt/G_O \leftarrow \left(\begin{smallmatrix} *&*&*\\ *&*&*\\ t*&*&*\\ \end{smallmatrix}\right)dt/(I \cap J) \rightarrow \left(\begin{smallmatrix} *&t^{-1}*&t^{-2}*\\ *&*&t^{-1}*\\ t*&*&*\\ \end{smallmatrix}\right)dt/J  \quad \text{for} \quad * \in O.\] \end{example}

\begin{remark}
Theorem \ref{AffineKernel} implies that the kernel of Iwahori--Whittaker averaging is characterized by the naive condition of whether the singular support is killed by transportation across a certain Lagrangian correspondence.

Indeed let $J^- \coloneqq \Ad_{t^{-\rho}}(I)$. Recall that $T^*((J^-, \chi) \backslash Y)$ is defined as a shifted Hamiltonian reduction of $T^*Y$. Iwahori--Whittaker averaging is related to a Lagrangian correspondence between $T^*(G_O \backslash Y)$ and $T^*((J^-, \chi) \backslash Y)$,
that is obtained from the correspondence \[\frg_O dt/G_O \leftarrow \frg_O dt \cap (f dt + (\frj^-)^\perp) \rightarrow (f dt + (\frj^-)^\perp)/J^-\]
by pullback along the moment map $G_F \backslash T^*Y \rightarrow \frg_Fdt/G_F$. An element in $\frg_O dt$ is $G_O$-conjugate to an element in $f dt + (\frj^-)^\perp$ if and only if it is regular at the closed point of the formal disc. Therefore Theorem \ref{AffineKernel}\ref{AffineKernel1} is equivalent to the naive condition of being killed by the Lagrangian correspondence.
\end{remark}

\section{The Iwahori--Whittaker category}
Here we show that Iwahori--Whittaker averaging sends standard (respectively costandard) perverse sheaves in the Satake category to standard (respectively costandard) perverse sheaves in the Iwahori--Whittaker category. This will be used in the proof of Proposition \ref{Exactness}.

This section was originally written to extend certain arguments from \cite{BGMRR19} to the complex analytic setting, where the Artin--Schreier sheaf is unavailable. However \cite{Sa26} has since appeared, providing all of the necessary results. We still include this section as an alternative reference.

\subsection*{The Casselman--Shalika formula}
The following lemma is an adaptation of  Proposition 3.11 of \cite{BGMRR19}.

They use that the top compactly supported cohomology of an etale local system is free over the coefficient ring. We do not have the Artin--Schreier sheaf, so we will instead use that top Borel--Moore homology of an orientable manifold with boundary is a free module whose rank is independent of the coefficient ring.

\begin{enumerate}
\item[-] Let $\Gr^\lambda \subset \Gr$ be the spherical orbit indexed by $\lambda \in \Lambda^+$. Let $J_{\lambda!}$ and $J_{\lambda*} \in \Perv(G_O \backslash \Gr)$ be the standard and costandard objects (truncated to be perverse) supported on the closure of $\Gr^{\lambda}$.
\item[-] Let $Y^\lambda \subset \Gr$ be the $J$-orbit indexed by $\lambda \in \Lambda$. Let $\Delta_\lambda'$ and $\nabla_\lambda' \in \Perv(J \backslash \Gr)$ be the standard and costandard objects supported on $Y^\lambda$.
\end{enumerate}

\begin{lemma}\label{StandardCostandardModular}
Let $\lambda \neq \mu \in \Lambda^+$ and $k$ be a ring. Then \beq \label{CleanZero} \Hom^0(\Psi*J_{\lambda!}, \nabla_\mu' \otimes_\bZ k) \simeq 0.\eeq
\end{lemma}
\begin{proof}
First we will prove \eqref{CS}, identifying the desired Hom space with top Borel--Moore homology of a certain orientable manifold with boundary. Then, using that it vanishes for complex coefficients by \eqref{ComplexCoefficients}, we will deduce that it vanishes for all coefficients.

Recall that $J^- \coloneqq \Ad_{t^{-\rho}} I$. Let $\chi: J^- \rightarrow \bC$ be the additive character obtained by projecting to the abelianization of $\Ad_{t^{-\rho}} N$ then summing over the simple root spaces. Let $K \subset J^-$ be the kernel of $\chi$. 

Let $i: X^\mu \rightarrow \Gr$ be the $J^-$-orbit containing $t^{\mu}$. Because $\mu$ is dominant, the character of $J^-$ descends to a function $\chi: X^\mu \rightarrow \bC$.
Let $X^{\mu}_{\geq 0}$ (respectively $X^{\mu}_{> 0}$ and $X^{\mu}_{= 0}$) be the preimage under $\chi: X^\mu \rightarrow \bC$ of the complex numbers that have nonnegative (respectively positive and zero) real part.

Let $Z \subset \Gr^{\lambda} \cap X^{\mu}_{\geq 0}$ be the union of the singular locus of $\Gr^{\lambda} \cap X^{\mu}_{> 0}$ and the singular locus of $\Gr^{\lambda} \cap X^{\mu}_{= 0}$. Let $U \subset \Gr^{\lambda} \cap X^{\mu}_{\geq 0}$ be the open complement of $Z$.

Let $\omega_{\bC_{\geq 0}} \in \Shv(\bC)$ be the dualizing sheaf supported on the complex numbers that have nonnegative real part.

We will use the following observations.
\begin{enumerate}[label=(\roman*)]
\item\label{WhittakerCalculation1} The functor $\Psi * -: \Shv(G_O \backslash \Gr) \rightarrow \Shv(J \backslash \Gr)$ is t-exact by Lemma \ref{PsiLemma}.
\item\label{WhittakerCalculation2} The big tilting sheaf $\Xi' \in \Shv_\cN(J\backslash G_F/J)$ is self dual.
\item\label{WhittakerCalculation3}
There is an isomorphism \[\Xi' * \nabla_\mu' \simeq \Av_{J*}i_*\chi^! \omega_{\bC_{\geq 0}}[-\langle 2\check{\rho}, \mu + \rho \rangle].\]
\item\label{WhittakerCalculation4} 
There is a commutative triangle
\beq \label{AveragingTriangle} \begin{tikzcd}[column sep = tiny]
\Shv(G_O \backslash \Gr) \arrow[rr, "\Av_{K!}"] \arrow[dr, "\Av_{J!}"'] & & \Shv(K \backslash \Gr)\\
& \Shv(J \backslash \Gr). \arrow[ur, "\Av_{K!}"'] &
\end{tikzcd}\eeq
\item\label{WhittakerCalculation5}
The $J^-$-orbit $X^{\mu}$ is contained in the semi-infinite $N_F$-orbit through $t^\mu$. Hence Theorem 3.2 of \cite{MV07} implies that \[\dim(\Gr^{\lambda} \cap X^{\mu}) \leq \langle\check{\rho}, \lambda + \mu \rangle.\]
Moreover $Z$ has codimension $\geq 2$ in $\Gr^\lambda \cap X^{\mu}_{\geq 0}$.
\end{enumerate}
Write $j: \Gr^{\lambda} \rightarrow \Gr$ for the inclusion of the spherical orbit. 

For every ring $k$, we have
\beq\label{CS} \begin{aligned}\Hom^0&(\Psi*J_{\lambda!}, \nabla_\mu' \otimes_\bZ k) \\ 
&\simeq \Hom^{-\langle 2\check{\rho}, \lambda \rangle}(\Psi*j_! \bZ_{\Gr^\lambda}, \nabla_\mu' \otimes_\bZ k) & \ref{WhittakerCalculation1}\\ 
&\simeq \Hom^{-\langle 2\check{\rho}, \lambda - \rho \rangle}(\Av_{J!} j_! \bZ_{\Gr^\lambda}, \Xi' * \nabla_\mu' \otimes_\bZ k) & \ref{WhittakerCalculation2} \\
&\simeq \Hom^{- \langle 2\check{\rho}, \lambda + \mu \rangle}(\Av_{J!} j_! \bZ_{\Gr^\lambda}, \Av_{J*} i_* \chi^! \omega_{\bC_{\geq 0}} \otimes_\bZ k)  & \ref{WhittakerCalculation3}\\
&\simeq \Hom^{- \langle 2\check{\rho}, \lambda + \mu \rangle}(j_! \bZ_{\Gr^\lambda}, i_* \chi^! \omega_{\bC_{\geq 0}} \otimes_\bZ k) & \ref{WhittakerCalculation4} \\
&\simeq H^{- \langle 2\check{\rho}, \lambda + \mu \rangle}(\Gr^\lambda \cap X^{\mu}_{\geq 0}, \omega \otimes_\bZ k) \\
&\simeq H^{- \langle 2\check{\rho}, \lambda + \mu \rangle}(U, \omega \otimes_\bZ k), & \ref{WhittakerCalculation5}
\end{aligned}\eeq
which either vanishes or is the top Borel--Moore homology of $U$.

Now we prove that this Borel--Moore homology vanishes for any coefficient ring, using that it vanishes for complex coefficients together with dimension considerations.

For complex coefficients, using the Riemann-Hilbert correspondence and the exponential D-module, the proof of Corollary 3.6 of \cite{BGMRR19} gives \beq\label{ComplexCoefficients}\Hom(\Psi*J_{\lambda!}, \nabla_\mu' \otimes_\bZ \bC) \simeq 0.\eeq

Since $U$ is an orientable manifold with boundary, its top Borel--Moore homology is free and the rank is the number of top dimensional connected components of $U$ without boundary.
Therefore \eqref{CS} and \eqref{ComplexCoefficients} imply that $U$ has no $\langle\check{\rho}, \lambda + \mu \rangle$-dimensional connected components without boundary. Hence \eqref{CS} vanishes for every coefficient ring, giving the desired vanishing 
\eqref{CleanZero}.
\end{proof}

\subsection*{Averaging standard and costandard objects}
The following proposition is adapted from the proof of Theorem 3.9 of \cite{BGMRR19}. It says that Iwahori--Whittaker averaging sends standard (respectively costandard) perverse sheaves in the Satake category to standard (respectively costandard) perverse sheaves in the Iwahori--Whittaker category.

\begin{proposition}\label{StandardCostandard}
For all $\lambda \in \Lambda^+$, there are isomorphisms \[\Psi*J_{\lambda!} \simeq \Xi' * \Delta_\lambda' \quad \text{and} \quad \Psi*J_{\lambda*} \simeq \Xi' * \nabla_\lambda' \quad \text{in} \quad \Shv(J \backslash \Gr).\]
\end{proposition}
\begin{proof}
The $J$-orbit $Y^\lambda$ is open in the support of $\Av_{J*} J_{\lambda!}$, and the restriction $(\Av_{J*} J_{\lambda!}[\langle 2\check{\rho}, \rho \rangle])|_{Y^\lambda}$ is the constant sheaf in perverse degree.

Let $Z^\lambda \subset \Gr$ be the union over $w \in W$ of $Y^{w\lambda}$. Then $Z^\lambda$ is open in the support of $\Psi * J_{\lambda!}$, and \beq\label{XiPsiOpen} (\Xi' * \Delta_{\lambda}')|_{Z^\lambda} \simeq (\Psi * J_{\lambda!})|_{Z^\lambda}.\eeq
Since $\Xi' * \Delta_{\lambda}'$ is the $!$-extension of its restriction to $Z^\lambda$, we obtain a map
\beq \label{XiPsiMap}\Xi' * \Delta_{\lambda}' \rightarrow \Psi * J_{\lambda!}\eeq
whose restriction to $Z^\lambda$ is the isomorphism \eqref{XiPsiOpen}.

Lemma \ref{StandardCostandardModular} implies that $\Psi * J_{\lambda!}$ does not admit a simple quotient supported outside $Z^\lambda$. (Such a simple quotient would inject into $\nabla_\mu' \otimes_\bZ \bF_p$ for some prime $p$ and some $\mu$ not in the Weyl group orbit of $\lambda$.) Hence \eqref{XiPsiMap}
is surjective.

Equation \eqref{ComplexCoefficients} implies that \eqref{XiPsiMap} becomes an isomorphism after tensoring with the complex numbers. Moreover $\Xi' * \Delta_\lambda'$ admits no torsion perverse subsheaf. Hence \eqref{XiPsiMap} is injective.

Therefore $\Xi' * \Delta_{\lambda}' \simeq \Psi * J_{\lambda!}$. Verdier duality gives also $\Xi' * \nabla_\lambda' \simeq \Psi * J_{\lambda*}$ as desired.
\end{proof}

\begin{remark}
Proposition \ref{StandardCostandard} can be used to prove a t-exact equivalence between representations of $\sfG$ and Iwahori--Whittaker sheaves on the affine Grassmannian. This result already appeared in \cite{Sa26}, and it extends Theorem 3.9 of \cite{BGMRR19} to the complex analytic setting.
\end{remark}

\section{Hecke operators}
For each unramified point on the curve, there is a scheme acted on by the loop group whose quotient by the spherical subgroup is $\Bun$. The corresponding moment map is given by restriction of Higgs fields to the formal disc. This geometry gives rise to Hecke operators acting on automorphic sheaves.

Here we prove that the Hecke action of a central sheaf at a tamely ramified point is isomorphic to the action of the corresponding spherical perverse sheaf at a nearby unramified point. This will be used in the proof of Proposition \ref{Exactness}.

\subsection*{Hecke modifications}
Here we recall the construction of Hecke operators.

Let $X$ be a smooth projective complex curve, equipped with a finite subset $S \subset X$ where we will allow tame ramification.
Let $\Bun$ denote the moduli space of $G$-bundles on $X$ with $N$-reductions along $S$.

For simplicity we assume here that $S = \{s\}$, i.e. there is a single tamely ramified point.

Let $\Aut$ be the group ind-scheme of continuous automorphisms of $O$. Let $\Aut^+ \subset \Aut$ be the pro-algebraic subgroup scheme of pointed automorphisms (its complex points are the same but $\Aut/\Aut^+ \simeq \Spf O$). 

Let $\Hecke$ be the moduli space parameterizing: a point $x \in X$, $G$-bundles $E$ and $E'$ on $X$ both equipped with $N$-reductions at $s$, and an identification $E|_{X - x} \simeq E'|_{X - x}$. 

For $X' \subset X - s$, there are maps
\[\begin{tikzcd}[column sep = 2em]
\Bun & \arrow[l, "a \text{ } (\text{forget } x \text{, } E')"'] \Hecke|_{X'} \arrow[d, "r \text{ } (\text{restrict to disc})"] \arrow[r, "b \text{ }(\text{forget } E)"] & \Bun \times X' \\
& (G_O \rtimes \Aut^+) \backslash \Gr \times N \backslash G/N. & 
\end{tikzcd}\]

For $B \in \Shv((G_O \rtimes \Aut^+) \backslash \Gr)$ and $A \in \Shv(\Bun)$, let $(B \boxtimes \delta) \widetilde{\boxtimes} A \coloneqq r^* (B \boxtimes \delta) \otimes a^* A$ where $\delta \in \Shv_{\cN}(N \backslash G/N)$ is the monoidal unit. Define spherical Hecke operators in families by \beq \label{HeckeFamilies} B *^{X'} A \coloneqq b_*((B \boxtimes \delta) \widetilde{\boxtimes} A) \; \in \; \Shv(\Bun \times X').\eeq

Let $\Fl \coloneqq G_F/I$ be the enhanced affine flag variety.
At the tamely ramified point, choosing a coordinate gives a diagram
\[\begin{tikzcd}
\Bun & \arrow[l, "a"'] \Hecke|_s \arrow[d, "r"] \arrow[r, "b"] & \Bun \\
& I \backslash \Fl. & 
\end{tikzcd}\]

For $B \in \Shv_{\cN}(I \backslash \Fl)$ and $A \in \Shv(\Bun)$, define $B \widetilde{\boxtimes} A \coloneqq r^* B \otimes a^*A$. Define the affine Hecke operator by \[B *^s A \coloneqq b_*(B \widetilde{\boxtimes} A) \; \in \; \Shv(\Bun).\]

\subsection*{The Nadler--Yun theorem}
Here we recall the main result of \cite{NY19a}. It says that Hecke operators send sheaves with nilpotent singular support to sheaves that are locally constant in the curve direction and nilpotent in the $\Bun$ direction.

We explain the proof by using
Proposition \ref{SSConvolve} and taking into account automorphisms of the disc. Namely we interpret the desired singular support condition in terms of the moment map for the semidirect product of the loop group and the automorphism group of the disc.

For simplicity we assume here that $S$ is empty, i.e. there is no tame ramification.

If $x$ is an $A$-point of $X$, let $D_x \coloneqq \Spec O_x$ denote the formal completion of $X \times \Spec A$ along the graph $\Gamma_x$.

Let $Y$ be the scheme whose $A$-points parameterize: an $A$-point $x$ of $X$, a $G$-bundle $E$ on $X \times \Spec A$, a continuous isomorphism $\tau: A[[t]] \simeq O_x$ such that $\tau(0 \times \Spec A) = \Gamma_x$, and a trivialization of the $G$-bundle $\tau^*(E|_{D_x})$.

Then $G_F \rtimes \Aut$ acts on $Y$ such that
\begin{enumerate}
\item[-] $(G_O \rtimes \Aut^+) \backslash Y \simeq \Bun \times X$,
\item[-] $G_O \backslash G_F \times^{G_O \rtimes \Aut^+} Y \simeq \Hecke$,
\item[-] the Hecke action \eqref{HeckeFamilies} given by convolution.
\end{enumerate}
There are natural identifications
\begin{enumerate}
\item[-] $\Lie(\Aut^+)^{\perp} = (t^{-1}O/O)dt \; \subset \; \Lie(\Aut)^* = (F/O)dt$,
\item[-] $\Lie(G_O \rtimes \Aut^+)^{\perp}/G_O \rtimes \Aut^+ = (\frg_O/G_O \times (t^{-1}O/O))dt/\Aut^+$,
\item[-] $\alpha: T_0^* \Spec O = (t^{-1}O/O)dt$.
\end{enumerate}

The action of $G_F \rtimes \Aut$ on $Y$ induces a moment map \[\mu: T^*\Bun \times T^*X \rightarrow (\frg_O/G_O \times (t^{-1}O/O))dt/\Aut^+.\]
that sends
\[(\xi \in \Map(X, \frg/(G \times \bC^{\times}))_{\omega}, x \in X, \eta \in T_x^*X) \; \mapsto \; (\tau^*\xi|_{D_x}, \alpha \tau^* \eta).\]
This is only defined up to the adjoint $\Aut^+$-action, because it depends on a choice of continuous isomorphism $\tau: O \simeq O_x$.

Writing $\Nilp \subset T^* \Bun$ for the global nilpotent cone,
\begin{enumerate}
\item[($\star$)] the desired singular support condition $\Nilp \times T^*_X X \subset T^* \Bun \times T^*X$ is exactly the preimage of $(\cN_O/G_O \times 0)dt/\Aut^+$ under the above moment map $\mu$.
\end{enumerate}

The following is Theorem 6.2.2 of \cite{NY19a}.

\begin{proposition}\label{NadlerYun}
Let $B \in \Shv((G_O \rtimes \Aut^+) \backslash \Gr)$ and $A \in \Shv(\Bun)$.
\begin{enumerate}[label=(\alph*)]
\item\label{NadlerYun1} Suppose that $V \subset (\frg_O/G_O \times (t^{-1}O/O))dt/\Aut^+$ satisfies the property that \[(\Ad_{G_F \rtimes \Aut}V) \cap  ((\frg_O/G_O \times (t^{-1}O/O))dt/\Aut^+) \; \subset \; V.\] Then $\mu(\SSup(A) \times T^*_X X) \subset V$ implies that $\mu(\SSup(B*A)) \subset V$.
\item\label{NadlerYun2} If $\SSup(A) \subset \Nilp$ then $\SSup(B*A) \subset \Nilp \times T^*_X X$.
\end{enumerate}
\end{proposition}
\begin{proof}
The infinite dimensional case of Proposition \ref{SSConvolve} implies \ref{NadlerYun1}.

The moment map induced by $\Aut$ acting on $G_F$ is given by \[T^*G_F \simeq G_F \times \frg_F dt \rightarrow (F/O)dt, \quad (g, \eta) \mapsto \langle (dg/dt)g^{-1}, \eta \rangle.\]
Thus the adjoint action of $G_F \rtimes \Aut$ on its dual Lie algebra satisfies \[\Ad_g(\eta, \xi) = (\Ad_g \eta, \xi + \langle (dg/dt)g^{-1}, \eta \rangle) \quad \text{for} \quad g \in G_F \text{ and } (\eta, \xi) \in (\frg_F\times F/O)dt.\]
Therefore Lemma 5.2.5 of \cite{NY19a} implies that $V \coloneqq (\cN_O/G_O \times 0)dt/\Aut^+$ satisfies the hypothesis of \ref{NadlerYun1}. Therefore \ref{NadlerYun1} and ($\star$) imply \ref{NadlerYun2}.
\end{proof}


\subsection*{Gaitsgory's functor} 
Here we show that acting by a central sheaf \cite{Ga01} at a tamely ramified point is equivalent to acting by the corresponding spherical perverse sheaf at a nearby unramified point.

For simplicity we assume here that $S = \{s\}$, i.e. there is a single tamely ramified point.

Let $\sfG$ be the Langlands dual group defined over the integers. 
Denote the geometric Satake equivalence \cite{MV07} by \[\sfS: \Rep(\sfG) \xrightarrow{\sim}\Perv((G_O \rtimes \Aut^+)\backslash \Gr).\]
(Proposition A.1 of \cite{MV07} says that every $G_O$-equivariant perverse sheaf admits a unique $\Aut^+$-equivariant structure.)

Let $X' \subset X$ be an affine open subvariety of the curve containing the tamely ramified point $s$.

Let $\Gait$ be the moduli space parameterizing: a point $x \in X'$, a $G$-bundle $E$ on $X'$ equipped with an $N$-reduction at $s$, and a trivialization of $E|_{X' - s - x}$.
Then $\Gait|_{X' -s}$ is a $(\Gr \times G/N)$-bundle over $X'-s$, and the special fiber is $\Gait|_s \simeq \Fl$.
Taking nearby cycles gives Gaitsgory's functor \[\sfZ: \Rep(\sfG) \rightarrow \Perv_{\cN}(I \backslash \Fl).\]

Let $\Shv_{\Nilp}(\Bun)$ be the category of sheaves on $\Bun$ with singular support contained in the global nilpotent cone.

The following is from Section 2.3.4 of \cite{NY19b}.
In the proof we will relate $\Hecke$ to $\Gait$ by uniformizing $\Bun$, then use that nearby cycles commutes with proper pushforward. 

\begin{proposition}\label{Nearby}
Let $A \in \Perv_{\Nilp}(\Bun)$ and $V \in \Rep(\sfG)$. For $x \in X' - s$ there is an isomorphism \beq\label{SphericalCentral} \sfS(V) *^x A   \simeq \sfZ(V) *^s A \; \in \; \Shv_{\Nilp}(\Bun).\eeq
(The isomorphism depends on a choice of path from $s$ to $x$.)
\end{proposition}
\begin{proof}
We first claim that for $A \in \Perv(\Bun)$ we have
\beq\label{NbySplit} \psi((\sfS(V) \boxtimes \delta) \widetilde{\boxtimes} A) \simeq \sfZ(V) \widetilde{\boxtimes} A,\eeq
where $\psi$ denotes nearby cycles from $\Hecke|_{X' - s}$ to $\Hecke|_s$.

The proof of this claim is based on Section 4.3.2 of \cite{Ga01}. Namely we need to compare nearby cycles in the $\Hecke$ and $\Gait$ families. For this we will trivialize the universal $G$-bundle by pulling back along various smooth maps $W \rightarrow \Bun$, so that the resulting pullbacks of $\Hecke$ split as products $\Gait \times W$.\footnote{We cannot simply pull back along a single universal map $W \rightarrow \Bun$ because the group of automorphisms of a $G$-bundle on an affine curve is not pro-smooth.}

Let $U \subset \Bun$ be a quasi-compact open substack. Then by \cite{DS95} there exists a smooth surjection $W \rightarrow U$ such that the pullback of the universal $G$-bundle along $W \times X' \rightarrow \Bun \times X$ is trivializable, compatibly with the $N$-reduction at $s$.

The choice of such a trivialization gives a splitting \[\beta: \Hecke|_{X'} \times_{\Bun} W  \xrightarrow{\sim} \Gait \times W\] fitting into commutative diagrams
\[\begin{tikzcd}[column sep = small]
\Hecke|_{X'-s} \times_{\Bun} W \arrow[d, "v"'] \arrow[r, "\sim"', "\beta"]& \Gait|_{X'-s} \times W \arrow[d]   \\
\Hecke|_{X'-s} \arrow[r, "{(r, a)}"'] & ((G_O \rtimes \Aut^+) \backslash \Gr \times N \backslash G/N) \times \Bun  
 \end{tikzcd}\]
 
 \noindent and 
 \[\begin{tikzcd}[column sep = small]
\Hecke|_s \times_{\Bun} W \arrow[d, "v"'] \arrow[r, "\sim"', "\beta"] & \Fl \times W \arrow[d] \\
\Hecke|_s \arrow[r, "{(r, a)}"'] & I \backslash \Fl \times \Bun.
 \end{tikzcd}\]

\noindent Therefore over $X' - s$ we get an isomorphism \[v^*((\sfS(V) \boxtimes \delta) \widetilde{\boxtimes} A) \simeq \beta^*((\sfS(V) \boxtimes \delta)|_{\Gait|_{X'-s}} \boxtimes A|_W).\]
Moreover nearby cycles commutes with smooth pullback. Therefore over $s$ we get an isomorphism \beq \label{UniformizeNearby} v^* \psi((\sfS(V) \boxtimes \delta) \widetilde{\boxtimes} A) \simeq \beta^* (\sfZ(V)|_{\Fl} \boxtimes A|_W) \simeq v^*(\sfZ(V) \widetilde{\boxtimes} A).\eeq

A different trivialization of the pullback to $W \times X'$ of the universal $G$-bundle gives a different splitting \[\beta': \Hecke|_{X'} \times_{\Bun} W \xrightarrow{\sim} \Gait \times W.\] The two trivializations differ by a map $W \rightarrow \Map((X', s), (G, N))$. The splittings $\beta$ and $\beta'$ differ by the corresponding family of actions on $\Gait$.
Commutativity of
\[\begin{tikzcd}
v^* \psi((\sfS(V) \boxtimes \delta) \widetilde{\boxtimes} A) \arrow[d, "\sim"'] \arrow[r, "\sim"]& \beta^* (\sfZ(V)|_{\Fl} \boxtimes A|_W) \arrow[d, "\sim"] \\
{\beta'}^* (\sfZ(V)|_{\Fl} \boxtimes A|_W) \arrow[r, "\sim"'] \arrow[ur, "\sim"]& v^*(\sfZ(V) \widetilde{\boxtimes} A), 
\end{tikzcd}\]
implies that \eqref{UniformizeNearby} is independent of the choice of trivialization. 

Therefore, the two isomorphisms obtained by pulling back \eqref{UniformizeNearby} along the two projections $W \times_U W \rightarrow W$ coincide. 
Because $A$ is assumed perverse, Proposition 10.2.7 of \cite{KS90} implies that \eqref{UniformizeNearby} descends to an isomorphism over $U \times_{\Bun} \Hecke|_{X'}$. As $U$ varies, these assemble to the desired isomorphism \eqref{NbySplit}.

Having proved the claim, using that nearby cycles commutes with proper pushforward (and pushforward of weakly $T$-constructible sheaves along projection to the base of a $T$-torsor) gives \beq\label{NearbyCentral} \psi(\sfS(V) *^{X'-s} A) \simeq \sfZ(V) *^s A,\eeq
where $\psi$ denotes nearby cycles from $\Bun \times (X'-s)$ to $\Bun \times s$.

Now assume that $A \in \Perv_{\Nilp}(\Bun)$ has nilpotent singular support. Proposition \ref{NadlerYun} says that $\sfS(V)*^{X-s} A$ is locally constant in the curve directions and has nilpotent singular support in the $\Bun$ directions. Therefore \[\sfS(V) *^x A \simeq \psi(\sfS(V) *^{X'-s} A).\]
Combined with \eqref{NearbyCentral}, this gives the desired \eqref{SphericalCentral}.
\end{proof}

\begin{remark}
We expect that by using the alternative construction of the central functor provided by \cite{Zh14}, it is probably not necessary to assume that $A \in \Shv_{\Nilp}(\Bun)$ is perverse. Indeed Zhu's moduli space can be uniformized by trivializing the relevant $G$-bundle on the formal disc rather than trivializing it on an affine curve.
\end{remark}

\section{The global nilpotent cone}
Theorem \ref{AffineKernel} characterizes the kernel of Whittaker averaging on automorphic sheaves by a singular support condition that appears to depend on the choice of point on the curve. To explain why (for nilpotent sheaves) this condition is actually independent of the point, we will study the generically regular locus of the global nilpotent cone.

Roughly we will show that
\begin{enumerate}
\item[-] its connected components are indexed by dominant coweights,
\item[-] within each such connected component, the irreducible components are indexed by certain tuples of Weyl group elements.
\end{enumerate}
This extends results from Section 2.10.3 of \cite{BD91} to allow tame ramification.\footnote{Beilinson and Drinfeld study all irreducible components not just the generically regular ones.}

Moreover we will prove that these irreducible components satisfy the following homogeneity property. For every unramified point on the curve, every irreducible component of the generically regular locus contains a Higgs field that is regular at that point.

This will be used to prove the equivalence between conditions \ref{KernelUnmarked2} and \ref{KernelUnmarked3} in Theorem \ref{KernelUnmarked}, but it is not needed for any other results in this paper.
In particular this section is not logically necessary for the proof of Theorem \ref{CentralTilting}.

\subsection*{The generically regular locus}
Recall that
\[\Bun \coloneqq \Map((X, S), (\pt/G, \pt/N))\]
denotes the moduli space of $G$-bundles on $X$ with $N$-reductions at $S$.

The global nilpotent cone is the closed (nonreduced) substack \[\Nilp \coloneqq \Map((X, S), (\cN/(G \times \bC^{\times}), \frn/N))_{\omega(S)} \; \subset \; T^*\Bun.\]
Here $\bC^{\times}$ acts on the nilpotent cone $\cN$ by scaling. The subscript `$\omega(S)$' is shorthand for the fiber product `$\times_{\Map(X, \pt/\bC^{\times})} \omega(S)$', i.e. we insist that the line bundle induced by $\cN/(G \times \bC^{\times}) \rightarrow \pt/\bC^{\times}$ is the canonical bundle twisted by the divisor $S$.

The irregular locus of the global nilpotent cone is the closed substack \[\Nilp^{\irreg} \coloneqq \Map((X, S), (\cN^{\irreg}/(G \times \bC^{\times}), \frn^{\irreg}/N))_{\omega(S)}.\]

Now we recall the notion of generic maps. If $U \subset Y$ is an open substack, let $\Map^{\gen}(X, U \subset Y)$ be the prestack whose $M$-points are maps $\sigma: X \times M \rightarrow Y$ such that there exists an open $V \subset X \times M$ satisfying
\begin{enumerate}
\item[-] $V \times_M M' \subset X \times M'$ is dense for all $M' \rightarrow M$,
\item[-] $\sigma|_V$ factors through $U$.
\end{enumerate}

The open complement of $\Nilp^{\irreg}$ is the generically regular locus \[\Nilp^{\reg} \coloneqq \Map^{\gen}((X, S), (\cN^{\reg}/(G \times \bC^{\times}) \subset \cN/(G \times \bC^{\times}), \frn/N))_{\omega(S)}.\]

\subsection*{Borel reduction}
It will be useful to replace $\Nilp^{\reg}$ by a stack $\widetilde{\Nilp}^{\reg}$ that has the same field valued points, but has better deformation theory and admits a map to the configuration space.

The generically regular locus in the global nilpotent cone is the fiber product \[\Nilp^{\reg} \simeq \Nilp'^{\reg} \times_{(\frg/G)^S} (\frn/N)^S,\]
where \[\Nilp'^{\reg} \coloneqq \Map^{\gen}(X, \cN^{\reg}/(G \times \bC^{\times}) \subset \cN/(G \times \bC^{\times}))_{\omega(S)}.\]

Let $\bC^{\times}$ act on $\frn$ by scaling, and define \[\widetilde{\Nilp}^{\reg} \coloneqq \widetilde{\Nilp}'^{\reg} \times_{(\frg/G)^S} (\frn/N)^S,\]
where \[\widetilde{\Nilp}'^{\reg} \coloneqq \Map^{\gen}(X, \frn^{\reg}/(B \times \bC^{\times}) \subset \frn/(B \times \bC^{\times}))_{\omega(S)}.\]

The following lemma says that every generically regular Higgs field admits a compatible $B$-reduction. 

\begin{lemma}\label{Springer}
The natural map $\widetilde{\Nilp}'^{\reg}  \rightarrow \Nilp'^{\reg}$ induces a bijection on field valued points. 
\end{lemma}
\begin{proof}
This follows by the valuative criterion and properness of the Springer resolution \[\frn/(B \times \bC^{\times}) \rightarrow \cN/(G \times \bC^{\times}).\qedhere\]
\end{proof}

\subsection*{The Kazhdan action}
Now we will use the Kazhdan action to contract the generically regular locus of the global nilpotent cone to the configuration space of divisors.

If $\check{\alpha}$ is a root, let $\bC_{\check{\alpha}}$ be the corresponding representation of $T$. Let $I$ be the set of positive simple roots. 

Let $\bC^{\times}$ act on $B$ by $b \mapsto \Ad_{\rho(z)} b$ and on $\frn$ by $x \mapsto z^{-1}\Ad_{\rho(z)} x$. This gives the Kazhdan $\bC^{\times}$-action on $\frn/(B \times \bC^{\times})$ that contracts it to $(\underset{\check{\alpha} \in I}{\Pi} \bC_{\check{\alpha}})/(T \times \bC^{\times})$.

Let $\Conf$ be the configuration space of positive coweight valued divisors. Because $G$ is adjoint \[\Conf \simeq \Map^{\gen}(X,  (\underset{\check{\alpha} \in I}{\Pi} \bC_{\check{\alpha}}^{\times})/(T \times \bC^{\times}) \subset (\underset{\check{\alpha} \in I}{\Pi} \bC_{\check{\alpha}})/(T \times \bC^{\times}))_{\omega(S)}.\]

\begin{lemma}\label{NilpRegComponents}
The natural map 
\beq\label{ToConf} \widetilde{\Nilp}'^{\reg} \rightarrow \Conf\eeq
induces a bijection on connected components.
\end{lemma}
\begin{proof}
The Kazhdan action induces a $\bC^{\times}$-action on $\widetilde{\Nilp}'^{\reg}$ that contracts it to $\Conf$. The map in question is the projection to the fixed locus of this action, so it induces a bijection on connected components.
\end{proof}

Thus the connected components of $\widetilde{\Nilp}'^{\reg}$ are labeled by dominant coweights.
\begin{enumerate}
\item[-] Let $\Conf^{\lambda} \subset \Conf$ be the connected component consisting of divisors of degree $\lambda \in \Lambda^+$.
\item[-] Let $\widetilde{\Nilp}'^{\reg, \lambda} \subset \widetilde{\Nilp}'^{\reg}$ be the preimage of $\Conf^{\lambda}$.
\item[-] Let $\widetilde{\Nilp}^{\reg, \lambda} \coloneqq \widetilde{\Nilp}'^{\reg, \lambda} \times_{(\frg/G)^S} (\frn/N)^S$.
\end{enumerate}

\subsection*{Smoothness}
To determine the irreducible components of $\widetilde{\Nilp}^{\reg}$, we will need the following smoothness.

\begin{lemma}\label{Smooth}
The residue map $\widetilde{\Nilp}'^{\reg, \lambda} \rightarrow (\frn/B)^S$ is smooth and has connected fibers.
\end{lemma}
\begin{proof}
Let $F$ be a $B$-bundle on $X$, and $\sigma \in H^0(X, \frn_F \otimes \omega(S))$ be a generically regular section. The relative tangent complex \[T_{(F, \sigma)}(\widetilde{\Nilp}'^{\reg}/(\frn/B)^S) \simeq \Gamma(X, \frb_F(-S)[1] \xrightarrow{\Ad_{\sigma}} \frn_F \otimes \omega)\] is concentrated in degrees $\leq 0$, because generic regularity of $\sigma$ implies that the cokernel of $\Ad_{\sigma}$ is a torsion sheaf. This implies the desired smoothness.

The Kazhdan action induces $\bC^{\times}$-actions such that $\widetilde{\Nilp}'^{\reg, \lambda} \rightarrow (\frn/B)^S$ is equivariant. For each $\bC^{\times}$-fixed point in $(\frn/B)^S$, the $\bC^{\times}$-fixed locus in its fiber is connected. (Indeed these $\bC^{\times}$-fixed loci in the fibers  are configuration spaces of divisors with prescribed behavior along $S$.) Therefore smoothness and Lemma \ref{ConnectedFibers} imply the desired connectedness of fibers.
\end{proof}

For connectedness of the fibers we used the following lemma.

\begin{lemma}\label{ConnectedFibers}
Let $Y$ and $Z$ be schemes both equipped with contracting $\bC^{\times}$-actions. Let $f:Y \rightarrow Z$ be a flat $\bC^{\times}$-equivariant morphism with reduced fibers. Assume that for each $\bC^{\times}$-fixed point of $Z$, the $\bC^{\times}$ fixed locus of its fiber is connected. Then all fibers of $f$ are connected.
\end{lemma}
\begin{proof}
Because the $\bC^{\times}$-action is contracting, every point in $Z$ is in the image of some $\bC^{\times}$-equivariant map $\bC \rightarrow Z$. By pulling back along such maps, it suffices to assume that $Z = \Spec \bC[t]$ equipped with the scaling $\bC^{\times}$-action.

By assumption the $\bC^{\times}$-fixed locus of $f^{-1}(0)$ is connected. Therefore $Y$ is also connected, because the $\bC^{\times}$-action contracts it to a connected scheme.

Now we will use the following argument from \cite[\href{https://stacks.math.columbia.edu/tag/055J}{055J}]{Stacks} to prove that the fibers are connected.

Let $x \in H^0(f^{-1}(\bC^{\times}), \cO)$ be a nonzero idempotent. Let $V = \Spec A$ be an affine open subscheme of $Y$. By flatness, $A$ is a subring of $A[t^{-1}]$. Let $n \geq 0$ be minimal such that $t^n x|_V \in A$. Using that $(t^nx)^2 = t^n(t^n x)$ and $f^{-1}(0)$ is reduced, we get that $n = 0$. We have proved $x|_V \in A$ for every affine open subscheme of $Y$. Therefore $x$ extends to a function on all of $Y$.

Connectedness of $Y$ implies that $x = 1$ is the trivial idempotent.  Thus we have proved that $f^{-1}(\bC^{\times})$ is connected. This implies the desired connectedness of fibers.
\end{proof}

\subsection*{Irreducible components}
Here we show that the irreducible components of $\widetilde{\Nilp}^{\reg}$ are indexed by dominant coweights and certain tuples of Weyl group elements.

Let $\St \coloneqq \frn/B \times_{\frg/G} \frn/N$, a version of the Steinberg stack. Its irreducible components are indexed by the Weyl group $W$. If $\underline{w} \in W^S$ is an $S$-tuple, let $\St_{\underline{w}} \subset \St^S$ be the corresponding irreducible component. 

\begin{proposition}\label{Irreducible}
For all $\lambda \in \Lambda^+$ and $\underline{w} \in W^S$, the stack \[\widetilde{\Nilp}^{\reg, \lambda, \underline{w}} \coloneqq \widetilde{\Nilp}^{\reg, \lambda} \times_{\St^S} \St_{\underline{w}}\] 
is irreducible (or empty). Moreover every irreducible component of $\widetilde{\Nilp}^{\reg}$ is of this form.
\end{proposition} 
\begin{proof}
Lemma \ref{Smooth} and \cite[\href{https://stacks.math.columbia.edu/tag/004Z}{004Z}]{Stacks} imply the desired irreducibility.
Moreover every irreducible component of $\widetilde{\Nilp}^{\reg}$ is of the desired form, because $\widetilde{\Nilp}^{\reg, \lambda, \underline{w}} \subset \widetilde{\Nilp}^{\reg}$ are closed irreducible substacks whose union is the entire $\widetilde{\Nilp}^{\reg}$.
\end{proof}

The following lemma characterizes when $\widetilde{\Nilp}^{\reg, \lambda, \underline{w}}$ is nonempty. 

\begin{lemma}\label{NilpRegIrred}
The following are equivalent
\begin{enumerate}[label=(\alph*)]
\item\label{NilpRegIrred1} $\widetilde{\Nilp}^{\reg, \lambda, \underline{w}}$ is nonempty,
\item\label{NilpRegIrred2} there exists a coweight valued divisor $D \in \Conf^{\lambda}$ such that $\langle \check{\alpha}, D_s \rangle \neq 0$ for every tamely ramified point $s \in S$ and every positive simple root $\check{\alpha} \in I$ satisfying $w_s \check{\alpha} \leq 0$.\footnote{Here $D_s \in \Lambda^+$ denotes the component of the divisor supported at $s \in S$. And $w_s \in W$ denotes the component of the tuple $\underline{w} \in W^S$ indexed by $s \in S$.}
\end{enumerate} 
\end{lemma}
\begin{proof}
We need the following facts about the Steinberg stack. For $w \in W$, define the orbital variety $O_w \subset \frn$ as the closure of $\Ad_B(\frn \cap \Ad_w(\frn))$.
\begin{enumerate}[label=(\roman*)]
\item\label{Orbital1} The image of $\St_w \rightarrow \frn/B$ is $O_w/B$.
\item\label{Orbital2} For every element of $O_w$, its image under $\frn \rightarrow \underset{\check{\alpha} \in I}{\Pi} \bC_{\check{\alpha}}$ vanishes for every positive root $\check{\alpha}$ satisfying $w^{-1} \check{\alpha} \leq 0$.
\item\label{Orbital3} If an element of $\underset{\check{\alpha} \in I}{\Pi} \bC_{\check{\alpha}}$ vanishes for every positive root $\check{\alpha}$ satisfying $w^{-1} \check{\alpha} \leq 0$, then its image under $\underset{\check{\alpha} \in I}{\Pi} \bC_{\check{\alpha}} \rightarrow \frn$ contained in $O_w$.
\end{enumerate}

Now we prove that \ref{NilpRegIrred1} implies \ref{NilpRegIrred2}. If a Higgs field in $\widetilde{\Nilp}'^{\reg}$ can be extended to $\widetilde{\Nilp}^{\reg, \lambda, \underline{w}}$ then by \ref{Orbital1} and \ref{Orbital2} its image under \eqref{ToConf} gives a divisor satisfying \ref{NilpRegIrred2}. 

Finally we prove that \ref{NilpRegIrred2} implies \ref{NilpRegIrred1}. For this let $D$ be a divisor satisfying \ref{NilpRegIrred2}. The inclusion of the fixed locus of the Kazhdan action \[\Conf \rightarrow \widetilde{\Nilp}'^{\reg}\]
gives a section of \eqref{ToConf} sending $D$ to a Higgs field in $\widetilde{\Nilp}'^{\reg, \lambda}$. By \ref{Orbital1} and \ref{Orbital3}, the image of this Higgs field under $\widetilde{\Nilp}'^{\reg} \rightarrow (\frn/B)^S$ is contained in the image of $\St_{\underline{w}}$. Therefore it can be extended to a Higgs field in $\widetilde{\Nilp}^{\reg, \lambda, \underline{w}}$.
\end{proof}

\begin{example}
Every element of the $\lambda = 0$ connected component is everywhere regular. Therefore the image of $\widetilde{\Nilp}'^{\reg, 0} \rightarrow (\frn/B)^S$ is contained in the generically regular locus. Hence $\widetilde{\Nilp}^{\reg, 0, \underline{w}}$ is empty unless all components of the tuple $\underline{w} \in W^S$ are the identity. 
In contrast, if $\lambda$ is sufficiently dominant, then $\widetilde{\Nilp}^{\reg, \lambda, \underline{w}}$ is nonempty for all $\underline{w} \in W^S$.
\end{example}

\subsection*{Homogeneity}
The following proposition says that every irreducible component of the generically regular locus of the global nilpotent cone satisfies a certain homogeneity property.

The idea is to relate Higgs fields to coweight valued divisors, such that a divisor supported away from a point gives a Higgs field that is regular at that point.
\begin{proposition}\label{Homogeneous}
For every unramified point $x \in X - S$, every irreducible component of $\Nilp^{\reg}$ contains a Higgs field that is regular at $x$.
\end{proposition}
\begin{proof}
Let $Y \subset \Nilp^{\reg}$ be an irreducible component. Write $Z \subset \Nilp^{\reg}$ for the union of all other irreducible components. Lemma \ref{Springer} implies that \[\widetilde{\Nilp}^{\reg} \rightarrow \Nilp^{\reg}\] is a bijection on field valued points.
Therefore the preimages of $Y$ and $Z$ are proper closed substacks whose union is the entire $\widetilde{\Nilp}^{\reg}$. Hence the preimage of $Y$ contains the entirety of some irreducible component of $\widetilde{\Nilp}^{\reg}$. Proposition \ref{Irreducible} implies that the preimage of $Y$ contains some $\widetilde{\Nilp}^{\reg, \lambda, \underline{w}}$.

By Lemma \ref{NilpRegIrred}, there exists a divisor $D \in \Conf^{\lambda}$ satisfying condition \ref{NilpRegIrred2}, such that $x$ does not appear in its support. Then
\[\Conf \rightarrow \widetilde{\Nilp}'^{\reg}\] sends $D$ to a Higgs field in $\widetilde{\Nilp}'^{\reg, \lambda}$. By \ref{Orbital1} and \ref{Orbital3} this can be extended to a Higgs field in $\widetilde{\Nilp}^{\reg, \lambda, \underline{w}}$ that is regular at $x$. Its image in $Y \subset \Nilp^{\reg}$ is also regular at $x$.
\end{proof}

\begin{remark}
In fact Footnote 17 of \cite{FR25} shows that $\widetilde{\Nilp}^{\reg} \rightarrow \Nilp^{\reg}$ is a locally closed embedding. Therefore it induces a homeomorphism between the underlying topological spaces and a bijection between the sets of irreducible components.
\end{remark}

\section{Whittaker averaging of automorphic sheaves}
Here we characterize the kernel of Whittaker averaging at unramified and tamely ramified points of the curve in terms of singular support. Using this we prove certain t-exactness properties of the Hecke action of Whittaker averaged central sheaves.

\subsection*{Anti-temperedness}
The following proposition says that the kernel of Whittaker averaging at an unramified point consists of exactly the sheaves that have irregular singular support.

Recall that from Remark \ref{KernelPsi} that convolving by $\Psi \in \Perv(J \backslash \Gr)$ has the same kernel as Iwahori--Whittaker averaging.

\begin{theorem}\label{KernelUnmarked}
Let $A \in \Shv(\Bun)$ and $x \in X - S$ be an unramified point. Then the following are equivalent
\begin{enumerate}[label=(\alph*)]
\item\label{KernelUnmarked2} every Higgs field in $\SSup(A)$ is irregular at $x$,
\item\label{KernelUnmarked1} $\Psi *^x A \simeq 0$, i.e. $A$ is killed by Whittaker averaging at $x$.
\end{enumerate}
If $\SSup(A) \subset \Nilp$, then the above is also equivalent to
\begin{enumerate}[label=(\alph*)]
\setcounter{enumi}{2}
\item\label{KernelUnmarked3} every Higgs field in $\SSup(A)$ is irregular everywhere.
\end{enumerate}
\end{theorem}
\begin{proof}
Theorem \ref{AffineKernel} implies \ref{KernelUnmarked2} is equivalent to \ref{KernelUnmarked1}. Moreover \ref{KernelUnmarked3} directly implies \ref{KernelUnmarked2}. 

It remains to prove that \ref{KernelUnmarked2} implies \ref{KernelUnmarked3}, i.e. independence of the point. The idea is that, although the singular support may contain a Higgs field that is generically regular but irregular at $x$, by homogeneity it would then necessarily contain a different Higgs field regular at $x$. 

Suppose for contradiction that $A \in \Shv_{\Nilp}(\Bun)$ contains a generically regular Higgs field. 
Recall that
\begin{enumerate}[label=(\roman*)]
\item\label{ComponentDimension1} $\Nilp^{\reg} \subset \Nilp$ is open, 
\item\label{ComponentDimension2} $\Nilp^{\reg}$ and $\SSup(A) \subset T^*\Bun$ are both half dimensional.\footnote{See \cite{Gi01} for a proof that the global nilpotent cone is Lagrangian. (We only need that the generically regular locus is Lagrangian, so Lemma \ref{Springer} can be used instead of Lemma 6 of \cite{Gi01}.)}
\end{enumerate}
Therefore there exists an irreducible component $Y \subset \Nilp^{\reg}$ such that \[\dim(\SSup(A) \cap Y) = \dim(\SSup(A) \cap \Nilp^{\reg}) \overset{\ref{ComponentDimension1}}{=}  \dim(\SSup(A)) \overset{\ref{ComponentDimension2}}{=} \dim Y.\]
Since $\SSup(A) \cap Y$ is closed in $Y$ of equal dimension, we get $Y \subset \SSup(A)$.
Proposition \ref{Homogeneous} now implies that $\SSup(A)$ contains a Higgs field that is regular at $x$.
\end{proof}

\begin{remark}
The property described in Theorem \ref{KernelUnmarked} is called `anti-temperedness'. It matches the notion introduced in \cite{AG15} by Corollary 2.5.12 of \cite{Be21}.
\end{remark}

\begin{remark}
Taking $H$ to be the full loop group in Proposition \ref{SSConvolve} gives an easier proof that \ref{KernelUnmarked3} implies \ref{KernelUnmarked1}, because $\cN^{\irreg}_Fdt$ is stable under $G_F$-conjugation. But to prove that \ref{KernelUnmarked2} implies \ref{KernelUnmarked1}, we needed to take $H$ to be a finite dimensional quotient of $\Ad_{t^{-\rho}}(G_O)$ and use the more involved analysis from Lemma \ref{RhoGIntersect}.
\end{remark}

\begin{remark}
In the setting of constructible nilpotent sheaves with complex coefficients (and no tame ramification), Færgeman and Raskin gave a different proof that  \ref{KernelUnmarked1} implies \ref{KernelUnmarked3} using the global Whittaker coefficients functor \cite{FR25}.
\end{remark}

\begin{remark}
For nilpotent sheaves, it is possible to prove that \ref{KernelUnmarked1} is independent of the choice of point using \cite{NY19a, Be21} and the argument sketched in Section 1.3 of \cite{FR23}. After proving that \ref{KernelUnmarked1} is equivalent to \ref{KernelUnmarked2}, this gives a different argument that \ref{KernelUnmarked2} implies \ref{KernelUnmarked3}, without using the geometry of the global nilpotent cone.
\end{remark}

\begin{remark}
In the setting of Zariski constructible sheaves with complex coefficients (and no tame ramification), Færgeman and Raskin proved that \ref{KernelUnmarked1} is independent of the choice of point \cite{FR23} using the de Rham spectral action. Therefore \ref{KernelUnmarked2} is equivalent to \ref{KernelUnmarked3} even without the assumption of nilpotent singular support.
\end{remark}

Let $\Xi \in \Shv_{\cN}(I \backslash \Fl)$ be the big tilting sheaf supported on the closure of the orbit indexed by the longest element $w_0 \in W$.

The following proposition describes the kernel of Whittaker averaging at a tamely ramified point in terms of singular support.

\begin{proposition}\label{KernelMarked}
Let $A \in \Shv(\Bun)$ and $s \in S$ be a tamely ramified point. Suppose that the residue at $s$ of every Higgs field in $\SSup(A)$ is nilpotent. Then the following are equivalent
\begin{enumerate}[label=(\alph*)]
\item the residue at $s$ of every Higgs field in $\SSup(A)$ is irregular,
\item $\Xi *^s A \simeq 0$, i.e. $A$ is killed by Whittaker averaging at $s$.
\end{enumerate}
\end{proposition}
\begin{proof}
Immediate by Theorem \ref{GeneralKernel}.
\end{proof}

The following corollary says that any automorphic sheaf that is killed by Whittaker averaging at every unramified point is also killed by Whittaker averaging at every tamely ramified point.

\begin{corollary}\label{KernelContainment}
Suppose that $A \in \Shv(\Bun)$ satisfies $\Psi *^x A 
\simeq 0$ for all $x \in X - S$. Then also $\Xi *^s A \simeq 0$ for all $s \in S$.
\end{corollary}
\begin{proof}
Suppose that an automorphic sheaf is killed by Whittaker averaging at every unramified point. Then Theorem \ref{KernelUnmarked} implies that every Higgs field in its singular support is irregular along $X - S$ and hence also irregular along $S$.
Therefore Proposition \ref{KernelMarked} implies that it is also killed by Whittaker averaging at every tamely ramified point.
\end{proof}

\subsection*{T-exactness}
Here we prove certain t-exactness properties of the Hecke action of Whittaker averaged central sheaves.

Write `$A \leq 0$' as shorthand for `$A$ lies in nonpositive degrees with respect to the perverse t-structure', and similarly for nonnegative degrees.

\begin{proposition}\label{Exactness}
Let $V \in \Rep(\sfG)$ admit a costandard (respectively standard) filtration, and $s \in S$ be a tamely ramified point. Then the functor \[\Xi *^s \sfZ(V) *^s -: \Shv_{\Nilp}(\Bun) \rightarrow \Shv_{\Nilp}(\Bun)\]
is right (respectively left) t-exact
\end{proposition}
\begin{proof}
Suppose that $V$ admits a costandard filtration. Proposition 13.1 of \cite{MV07} and Proposition \ref{StandardCostandard} imply that $\Psi * \sfS(V) \in \Shv(J \backslash \Gr)$ also admits a costandard filtration. Therefore, for every unramified point $x \in X - S$ the functor \beq\label{ConvolveExact} \Psi *^x \sfS(V) *^x -  \quad \text{is right t-exact.}\eeq

Suppose that $A \in \Perv_{\Nilp}(\Bun)$ is perverse. For every unramified point $x \in X - S$, t-exactness of $\Psi *^x -$ implies \[\Psi *^x\tau^{>0} (\sfZ(V) *^s A) \simeq \tau^{>0}(\Psi *^x \sfZ(V) *^s A) \overset{\text{Proposition }\ref{Nearby}}{\simeq} \tau^{>0}(\Psi *^x \sfS(V)*^x A) \overset{\eqref{ConvolveExact}}{\simeq} 0.\] Therefore Corollary \ref{KernelContainment} implies $\Xi *^s \tau^{>0} (\sfZ(V) *^s A) \simeq 0$. The t-exactness of $\Xi *^s -$ implies that $\Xi *^s \sfZ(V) *^s A \leq 0$. We have proved that $\Xi *^s \sfZ(V) *^s -$ is right t-exact.

If $V$ admits a standard filtration, a similar argument shows that $\Xi *^s \sfZ(V) *^s -$ is left t-exact.
\end{proof}

\begin{remark}
Let us explain the proof of Proposition \ref{Exactness} more informally. If $V$ admits a costandard filtration, then Propositions \ref{StandardCostandard} and \ref{Nearby} imply that $\sfZ(V) *^s -$ is right t-exact modulo the kernel of Whittaker averaging at every unramified point. Corollary \ref{KernelContainment} then implies that it is also right t-exact modulo the kernel of Whittaker averaging at every tamely ramified point.
\end{remark}

\begin{remark}
Before Whittaker averaging, the Hecke operator $\sfS(V) *^x -$ is typically not t-exact. Indeed it acts on the constant sheaf on $\Bun$ by tensoring it with $\Gamma(\Gr, \sfS(V))$, which is typically concentrated in many cohomological degrees.
\end{remark}

\section{Application to central sheaves}
Mirković showed that convolution exact sheaves in the finite Hecke category are tilting, as is explained in Remark 7 of \cite{AB09}. 
This does not directly generalize to the affine Hecke category because the affine Weyl group does not contain a longest element. Indeed central sheaves are convolution exact but (before Whittaker averaging) they are typically not tilting.

For the remainder of this paper, we specialize $X$ to be a genus zero curve with tame ramification at two points. Thus $\Bun$ now denotes the moduli stack of $G$-bundles on $\bP^1$ with an $N$-reduction at zero and an $N^-$-reduction at infinity.\footnote{Choosing an $N^-$-reduction rather than an $N$-reduction at infinity produces the same moduli stack, but is more convenient for indexing the orbits.}

The Radon transforms are two different equivalences between $\Shv_{\Nilp}(\Bun)$ and the affine Hecke category. Bezrukavnikov and Morton-Ferguson adapt Mirković's argument to the affine setting by using the Radon transforms as a substitute for the longest element of the Weyl group \cite{BMF24}. Namely they show that an object in the affine Hecke category is tilting if and only if both of its Radon transforms are perverse. By combining this with the t-exactness proved in Proposition \ref{Exactness}, we will deduce the tilting property of Whittaker averaged central sheaves.

\subsection*{The Radon transform}
Here we recall the definitions of the two Radon transform functors. We prove that both are equivalences. The first sends costandards to standards, whereas the second sends standards to costandards.

Let $W^{\aff} = W \ltimes \Lambda$ be the extended affine Weyl group.
\begin{enumerate}
\item[-] Let $\Delta_w$ and $\nabla_w \in \Perv_{\cN}(I \backslash \Fl)$ be the standard and costandard extension of the universal local system on the orbit indexed by $w \in W^{\aff}$.
\item[-] Let $\Delta_w^{\thick}$ and $\nabla_w^{\thick} \in \Perv_{\Nilp}(\Bun)$ be the standard and costandard extension of the universal local system on the orbit indexed by $w \in W^{\aff}$.
\item[-] Let $\Shv_{\cN}^{\cpt}(I \backslash \Fl)$ be the subcategory of compact objects.
\end{enumerate}
The following two subcategories are different from each other because $\Bun$ is not quasi-compact.
\begin{enumerate}
\item[-] Let $\Shv_{\Nilp}^!(\Bun) \subset \Shv_{\Nilp}(\Bun)$ be the subcategory generated under finite colimits by $\Delta_w^{\thick}$ (equivalently this is the subcategory of compact objects).
\item[-] Let $\Shv_{\Nilp}^*(\Bun) \subset \Shv_{\Nilp}(\Bun)$ be the subcategory generated under finite colimits by $\nabla_w^{\thick}$ (these costandard extensions are typically not compact).
\end{enumerate}

The following proposition adapts \cite{Yu09} to the universal monodromic setting.

\begin{proposition}\label{Radon}
There are equivalences 
\begin{enumerate}[label=(\alph*)]
\item\label{Radon1} $- *^0 \Delta_1^{\thick}: \Shv_{\cN}^{\cpt}(I \backslash \Fl) \xrightarrow{\sim} \Shv_{\Nilp}^!(\Bun)$ sending $\nabla_w *^0 \Delta_1^{\thick} \simeq \Delta_w^{\thick}$,
\item\label{Radon2} $- *^0 \nabla_1^{\thick}: \Shv_{\cN}^{\cpt}(I \backslash \Fl) \xrightarrow{\sim} \Shv_{\Nilp}^*(\Bun)$ sending $\Delta_w *^0 \nabla_1^{\thick} \simeq \nabla^{\thick}_w$.
\end{enumerate}
\end{proposition}
\begin{proof}
Let $I^{\glob} \coloneqq G[t^{-1}] \times_G B^-$, where $G[t^{-1}] \rightarrow G$ is evaluation at infinity.
We will now construct a $\bC^{\times}$-action on $G_F$ that contracts $I^{\glob}$ to the torus and repels $I$ from the identity. Let $T \times \bC^{\times}$ act on $G_F$ by the adjoint $T$-action and the loop rotation $\bC^{\times}$-action. Let $\bC^{\times}$ act on $G_F$ via the coweight $\bC^{\times} \rightarrow T \times \bC^{\times}$ sending $z \mapsto (\rho(z^{-1}), z^{-n})$, for $n$ greater than the height of the highest root. 

For $v \in W^{\aff}$, let $i: T \rightarrow \Bun$ be the inclusion of the corresponding $T$-orbit. Consider the commuting diagram
\[\begin{tikzcd}
 & \arrow[dl, "c"'] I^{\glob} \arrow[r, "b"', shift right] \arrow[d] & \arrow[l, "r"', shift right] T \arrow[d, "i"] \\
I \backslash \Fl & \arrow[l, "a"]\Fl \arrow[r, "b"'] & \Bun.
\end{tikzcd}\]

\noindent Observe that
\begin{enumerate}[label=(\roman*)]
\item\label{Contract0} the square in the above diagram is Cartesian,
\item\label{Contract1} $b_!$ and $r^!$ are isomorphic on weakly $\bC^{\times}$-constructible sheaves by the contraction principal, 
\item\label{Contract2} $c^*$ and $c^!$ are isomorphic up to a shift by Proposition 5.4.13 of \cite{KS90}, because the $I^{\glob}$-orbits are transverse to the $I$-orbits.
\end{enumerate}
Therefore
\beq\label{RadonCalc} i^* (\nabla_w *^0 \Delta_1^{\thick})  \simeq i^*b_!a^* \nabla_w \overset{\ref{Contract0}}{\simeq}  b_! c^* \nabla_w \overset{\ref{Contract1}, \ref{Contract2}}{\simeq} r^! c^! \nabla_w[2(\ell(v) - \ell(w))]\eeq
is the universal local system if $v = w$ and vanishes otherwise.
Hence \beq\label{RadonNabla} \nabla_w *^0 \Delta_1^{\thick} \simeq \Delta_w^{\thick}.\eeq

The Radon transform $- *^0 \Delta_1^{\thick}$ admits a right adjoint $a_*b^!$. A calculation similar to \eqref{RadonCalc} shows that this right adjoint sends $a_*b^! \Delta_w^{\thick} \simeq \nabla_w$.
Therefore $a_*b^!(- *^0\Delta_1^{\thick})$ is equivalent to the identity functor, because it is $\Shv_{\cN}^{\cpt}(I \backslash \Fl)$-linear and preserves the monoidal unit. Hence the Radon transform $- *^0 \Delta_1^{\thick}$ is fully faithful. Moreover it is also essentially surjective by \eqref{RadonNabla}. This completes the proof of \ref{Radon1}. The proof of \ref{Radon2} is similar.
\end{proof}

\subsection*{Universal perversity}
Let $R \coloneqq \bZ[\Lambda]$ be the group ring of the coweight lattice. Regard $\Shv_{\Nilp}(\Bun)$ as linear over $R$ using the Hecke action at infinity.

\begin{enumerate}
\item[-] Let $\langle \Delta^{\thick}[\geq 0] \rangle \subset \Shv_{\Nilp}^!(\Bun)$ be the subcategory generated under extensions by $\Delta^{\thick}_w[i]$ for $w \in W^{\aff}$ and $i \geq 0$.
\item[-] Let $\langle \nabla^{\thick}[\leq 0] \rangle \subset \Shv_{\Nilp}^*(\Bun)$ be the subcategory generated under extensions by $\nabla_w^{\thick}[i]$ for $w \in W^{\aff}$ and $i \leq 0$.
\end{enumerate}

\begin{lemma}\label{UniversalPerverse}
If $A \in \Shv_{\Nilp}^!(\Bun)$ and $B \in \Shv_{\Nilp}^*(\Bun)$ then 
\begin{enumerate}[label=(\alph*)]
\item $A \in \langle \Delta^{\thick}[\geq 0] \rangle$ if and only if $A \leq 0$,
\item $B \in \langle \nabla^{\thick}[\leq 0] \rangle$ if and only if $B \otimes_R M \geq 0$ for every $R$-module $M$.
\end{enumerate}
\end{lemma}
\begin{proof}
Suppose that $B \otimes_R M \geq 0$ for every $R$-module. If $w \in W^{\aff}$ then $\Hom(\Delta_w, B) \otimes_R M \geq 0$ for every $R$-module. Therefore that $\Hom(\Delta_w, B)$ is quasi-isomorphic to a complex of projective $R$-modules concentrated in degrees $\geq 0$. Moreover \cite{Sw78} says that all finitely generated projective $R$-modules are free. By the Cousin filtration $B \in \langle \nabla^{\thick}[\leq 0] \rangle$.
\end{proof}

\subsection*{Proof of the tilting property}
Here we prove the tilting property of Whittaker averaged central sheaves.

We say that a sheaf in $\Shv_{\cN}(I \backslash G_F/I)$ admits a standard (respectively costandard) filtration if it has a filtration whose graded pieces are isomorphic to $\Delta_w$ (respectively $\nabla_w$).
We call a sheaf `tilting' if it admits both standard and costandard filtrations.

\begin{enumerate}
\item[-] Let $\langle \nabla[\geq 0] \rangle$ be the subcategory of $\Shv_{\cN}(I \backslash \Fl)$ generated under extensions by $\nabla_w[i]$ for $w \in W^{\aff}$ and $i \geq 0$.
\item[-] Let $\langle \Delta[\leq 0] \rangle$ be the subcategory of $\Shv_{\cN}(I \backslash \Fl)$ generated under extensions by $\Delta_w[i]$ for $w \in W^{\aff}$ and $i \leq 0$.
\end{enumerate}

\begin{theorem}\label{CentralTilting}
Let $V \in \Rep(\sfG)$ and $w \in W^{\aff}$. We have
\begin{enumerate}[label=(\alph*)]
\item\label{CentralTilting1} if $V$ admits a standard filtration, then $\Xi * \sfZ(V) * \Delta_w$ admits a standard filtration,
\item\label{CentralTilting2} if $V$ admits a costandard filtration, then $\Xi * \sfZ(V) * \nabla_w$ admits a costandard filtration.
\end{enumerate}
In particular
\begin{enumerate}[label=(\alph*)]\setcounter{enumi}{2}
\item\label{CentralTilting3} if $V$ is tilting, then $\Xi * \sfZ(V)$ is tilting.
\end{enumerate}
\end{theorem}
\begin{proof}
First we prove \ref{CentralTilting1}. Suppose that $V$ admits a standard filtration.
For every $R$-module $M$, Proposition \ref{Exactness} and Lemma \ref{UniversalPerverse} imply \[\Xi *^0 \sfZ(V) *^0 \nabla_w^{\thick} \otimes_R M \; \in \; \langle \nabla^{\thick}[\leq 0]\rangle.\]
Therefore Proposition \ref{Radon} implies \beq\label{DeltaLeq} \Xi * \sfZ(V) * \Delta_w \; \in \; \langle \Delta[\leq 0] \rangle.\eeq

Moreover Theorem 1(a) of \cite{Ga01} implies that the central sheaf $\sfZ(V)$ is convolution exact. Since $\Xi$ is also convolution exact, $\Xi * \sfZ(V) * \nabla_w$ is perverse. Hence \beq\label{DeltaGeq}\Xi * \sfZ(V) * \Delta_w \; \in \; \langle \Delta[\geq 0] \rangle.\eeq
Combining \eqref{DeltaLeq} and \eqref{DeltaGeq} gives that $\Xi * \sfZ(V) * \Delta_w$ admits a standard filtration.

Now we prove \ref{CentralTilting2}. 
Suppose that $V$ admits a costandard filtration. Then Proposition \ref{Exactness} and Lemma \ref{UniversalPerverse} imply \[\Xi *^0 \sfZ(V) *^0 \Delta_w^{\thick} \; \in \; \langle \Delta^{\thick}[\geq 0] \rangle.\]
Therefore Proposition \ref{Radon} implies \beq\label{NablaGeq} \Xi * \sfZ(V) * \nabla_w \; \in \; \langle \nabla[\geq 0] \rangle.\eeq

Moreover $\Xi$ and $\sfZ(V)$ are both convolution exact. Therefore $\Xi * \sfZ(V) * \nabla_w \otimes_R M$ is perverse for every $R$-module $M$. Hence \beq\label{NablaLeq}\Xi * \sfZ(V) * \nabla_w \; \in \; \langle \nabla[\leq 0] \rangle.\eeq
Combining \eqref{NablaGeq} and \eqref{NablaLeq} gives that $\Xi * \sfZ(V) * \nabla_w$ admits a costandard filtration.
\end{proof}


\begin{remark}
Theorem \ref{CentralTilting}\ref{CentralTilting3} was proved for field coefficients in \cite{AB09, BRR20}, with the exception of some very small primes. Their method of reduction to quasi-minuscule representations does not seem to lead to a proof of Theorem \ref{CentralTilting}\ref{CentralTilting1} and \ref{CentralTilting2}.
\end{remark}

\begin{remark}\label{FiniteExact}
Theorem \ref{CentralTilting}\ref{CentralTilting1} can be used to show that Bezrukavnikov's equivalence becomes t-exact after restriction to the finite Hecke category, generalizing Corollary 42(a) of \cite{Be16} to integer coefficients. Theorem \ref{CentralTilting}\ref{CentralTilting3} does not seem to suffice for this, because every $\sfG$-module admits a nonozero map from a standard module but not necessarily from a tilting module.
\end{remark}

\end{document}